\documentclass{llncs}
\usepackage{amsmath,amssymb,amsfonts,mathrsfs,thm-restate,thmtools}

\usepackage{subcaption}
\usepackage{float}
\usepackage{caption}
\usepackage{stfloats}
\usepackage{graphicx}
\usepackage{appendix}
\usepackage{epsf}
\usepackage{epsfig}
\usepackage{epstopdf}
\usepackage{xspace,xcolor}
\usepackage{soul}
\usepackage{comment}
\usepackage{latexsym}
\usepackage{tikz}
\usepackage{framed}
\usepackage{todonotes}
\usepackage{lipsum}
\usepackage{setspace}
\usepackage{xcolor}
\usepackage{tabulary}
\usepackage{tabularx}
\usepackage{multirow}
\usepackage{booktabs}

\usepackage{algpseudocode}

\usepackage[linesnumbered,ruled,vlined]{algorithm2e}
\usepackage{mathtools}

\usepackage{caption}
\usepackage{subcaption}
\usepackage{tkz-tab}
\usetikzlibrary{arrows,positioning}
\usetikzlibrary{shapes,snakes}

\usepackage[bookmarks=false, colorlinks,
linkcolor=red!60!black, citecolor=green!60!black]{hyperref}
\usepackage{nameref}

\usepackage{enumerate}
\usepackage{caption}
\usepackage{subcaption}
\usepackage[compatibility=false]{caption}
\tikzstyle{vertex}=[circle, draw, inner sep=0pt, minimum size=4.5pt]

\title{Star Coloring of Tensor Product of Two Graphs}

\author{Harshit Kumar Choudhary, 
Swati Kumari
  and  I. Vinod Reddy}

\institute{
Department of Computer Science and Engineering,  IIT Bhilai, India\\
\email{harshitk@iitbhilai.ac.in, swatik@iitbhilai.ac.in, vinod@iitbhilai.ac.in}}

\begin{document}
	
	\pagestyle{plain}
	
	\maketitle
 \begin{abstract}
A star coloring of a graph $G$ is a proper vertex coloring such that no path on four vertices is bicolored. The smallest integer $k$ for which $G$ admits a star coloring with $k$ colors is called the star chromatic number of $G$, denoted as $\chi_s(G)$. In this paper, we study the star coloring of tensor product of two graphs and obtain the following results.

\begin{enumerate}
\item We give an upper bound on the star chromatic number of the tensor product of two arbitrary graphs.
\item We determine the exact value of the star chromatic number of tensor product two paths.
\item We show that the star chromatic number of tensor product of two cycles is five, except for $C_3 \times C_3$ and $C_3 \times C_5$. 

\item We give tight bounds for the star chromatic number of tensor product of a cycle and a path. 
\end{enumerate}
\end{abstract}

\section{Introduction}\label{S:intro}
A proper $k$-coloring of a graph $G$ is an assignment of colors to the vertices of $G$ from the set $\{1,2,\ldots,k\}$ such that no two adjacent vertices are assigned the same color. The smallest integer $k$ for which $G$ admits a  proper $k$-coloring is called the \emph{chromatic number} of $G$, denoted by $\chi(G)$.  A $k$-star coloring of a graph $G$ is a proper $k$-coloring of $G$ such that every path on four vertices uses at least three distinct colors. The smallest integer $k$ such that $G$ has a $k$-star coloring is 
called \emph{star chromatic number} of $G$, denoted by $\chi_s(G)$.

Star coloring of graphs was introduced by Gr{\"u}nbaum in~\cite{grunbaum1973acyclic}. The problem is NP-complete even when restricted to planar bipartite graphs~\cite{albertson2004coloring}  and line graphs of subcubic graphs~\cite{lei2018star}. The problem is polynomial time solvable on cographs~\cite{lyons2011acyclic}, line graphs of tress~\cite{omoomi2018polynomial}, outer planar graphs and 2-dimensional grids~\cite{fertin2004star}. 
 Recently  Shalu and Cyriac~\cite{shalu2022complexity} showed that for $k \in \{4,5\}$, the $k$-star coloring is {\sf NP}-complete for graphs of degree at most four.

The \emph{Cartesian product} and \emph{tensor product} of two graphs $G$ and $H$ are denoted by $G\square H$ and $G \times H$ respectively.  The vertex set of the above products is $V(G) \times V(H)$ and their edges are determined as follows. Let $(u,v), (u',v') \in V(G) \times V(H)$. Then $(u,v) (u',v')$ belongs to
\begin{enumerate}

\item $E(G \square H)$ if either $u=u'$ and $vv'\in E(H)$, or $v=v'$ and $uu'\in E(G)$.
\item $E(G \times H)$ if $uu'\in E(G)$ and $vv'\in E(H)$.
\end{enumerate}

Proper coloring has been well studied with respect to various graph products. The chromatic number of the Cartesian product of two graphs $G$ and $H$ is equal to the maximum of chromatic numbers of $G$ and $H$~\cite{sabidussi1957graphs}. The chromatic number of a lexicographic product of two graphs $G$ and $H$ is equal to the $b$-fold chromatic number of $G$, where $b=\chi(G)$~\cite{geller1975chromatic}. The chromatic number of tensor product of two graphs $G$ and $H$ is at most the chromatic numbers of graphs $G$ and $H$~\cite{shitov2019counterexamples}.

Star coloring of the Cartesian product of graphs has been studied in several papers~\cite{fertin2004star,han2016star,akbari2019star}. Fertin et al.~\cite{fertin2004star} established an upper bound on the star chromatic number of the Cartesian product of two arbitrary graphs. They gave exact values of the star chromatic number for the Cartesian product of two paths. Han et al.~\cite{han2016star} studied the star coloring of Cartesian products of paths and cycles and determined the star chromatic number for some of the cases. Extending this work, Akbari et al.~\cite{akbari2019star} studied the star coloring of the Cartesian product of two cycles. They showed that the Cartesian product of any two cycles except $C_3 \square C_3$ and $C_3 \square C_5$ has a $5$-star coloring. 
 
Motivated by the results obtained in ~\cite{fertin2004star,han2016star,akbari2019star}, in this paper we focus on star coloring of the tensor product of graphs. In Section~\ref{S:tensor}, we establish an upper bound on star chromatic number of tensor product of two arbitrary graphs. In Section~\ref{sec-PP}  we give exact values of star chromatic number of tensor product of two paths. In Section~\ref{sec-CC}, we study the star coloring of tensor product of two cycles. We showed that tensor product of two cycles except $C_3 \times C_3$ and $C_3 \times C_5$ has a $5$-star coloring. In Section~\ref{sec-CP}, we study the star coloring of tensor product of a cycle and path. 
In some cases, we give the exact value of the star chromatic number and in some cases we give upper bounds for the star chromatic number.

\section{Preliminaries}\label{S:prelims}
In this section, we introduce some basic notation and terminology related to graph theory that we need throughout the paper. 
All the graphs considered in this paper are undirected, finite and simple (no self-loops and no multiple edges). 
For a  graph $G=(V,E)$, by $V(G)$ and $E(G)$ we denote the vertex set and edge set of $G$ respectively. The set $\{1,2,\ldots,k\}$ is denoted by $[k]$. 
We use $P_n$ and $C_n$ to denote a path and a cycle on $n$ vertices respectively. We denote the complete bipartite graph using $K_{m,n}$. For any positive integer $n$, $K_{1,n}$ is called a star graph.

In the proofs of our results we use the following known results.

\begin{lemma}~\cite{fertin2004star}
For a positive integer $n$, where $n \geq 2$, we have 

$$
		\chi_s(P_n) = 
			\begin{cases} 
				2 &\text{if $n \in \{2,3\}$}; \\
				3 &\text{otherwise.} \\
			\end{cases}
$$	
\end{lemma}

\begin{lemma}~\cite{fertin2004star}\label{lem-c}
For a positive integer $n$, where $n \geq 3$, we have 

$$
		\chi_s(C_n) = 
			\begin{cases} 
				4 &\text{if $n =5$}; \\
				3 &\text{otherwise.} \\
			\end{cases}
$$	
\end{lemma}

\begin{lemma}\label{lem-subgraph}
    For any subgraph $H$ of a graph $G$, we have $\chi_s(H) \leq \chi_s(G)$.
\end{lemma}

The following result on star coloring
of the Cartesian product of two paths is used in our results.

\begin{lemma}\label{lem-1}\cite{fertin2004star}
For every pair of positive integers $m$ and $n$, where $2 \leq m \leq n$, we have
\[
\chi_s(P_m \square P_n) = 
\begin{cases}
    3, &\text{if $m=n=2$};\\
    4, &\text{if $m\in \{2,3\}$, $n \geq 3$};\\
    5, &\text{if $m\geq 4$, $n\geq 4$.}
\end{cases}
\]
    
\end{lemma}

We denote the graphs shown in the Fig.~\ref{fig:zy-graph} as $Z$-graph and $Y$-graph respectively. We found that $\chi_s(Z)=\chi_s(Y)=5$ by performing a tedious case-by-case analysis. This helps to establish the lower bounds in some cases. 

\begin{center}
\begin{figure}[!h] 
	\begin{subfigure}[t]{0.5\textwidth}		
			 \includegraphics[trim=7cm 21.5cm 8cm 2.9cm, clip=true, scale=0.8]{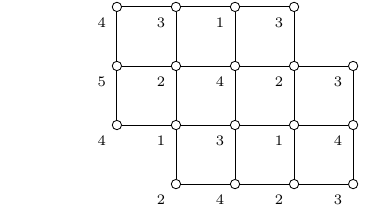}
    \label{fig:z-graph}		
	\end{subfigure}
	\begin{subfigure}[t]{0.5\textwidth}		
	  \includegraphics[trim=7cm 21.5cm 8cm 2.9cm, clip=true, scale=0.8]{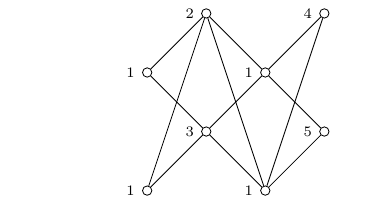}
    \label{fig:y-graph}
	\end{subfigure}	
	\caption{~Star coloring of $Z$-graph (left) and $Y$-graph (right). }
    \label{fig:zy-graph}

\end{figure}
\end{center}

A $k$-star coloring of $G \times H$ can be represented by a pattern (matrix) with $n_1$ rows and $n_2$ columns, where $n_1=|V(G)|$ and $n_2=|V(H)|$. For example, a $3$-star coloring of  $P_3 \times P_4$ can be represented by a pattern as shown in the Fig.~\ref{fig:prelim}.

\begin{center}
\begin{figure}[!h] 
	\begin{subfigure}[t]{0.5\textwidth}		
			 \includegraphics[trim=6cm 21.5cm 9cm 2.9cm, clip=true, scale=0.8]{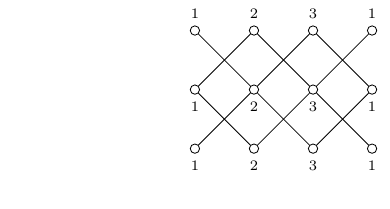}
	\end{subfigure}
	\begin{subfigure}[t]{0.5\textwidth}		
	  \includegraphics[trim=8cm 23cm 8cm 2.5cm, clip=true, scale=1]{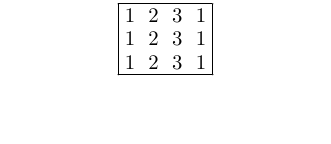}
	\end{subfigure}	
	\caption{~A 3-star coloring of $P_3 \times P_4$ (left) and coloring pattern representing 3-star coloring of $P_3 \times P_4$ (right). }
    \label{fig:prelim}

\end{figure}
\end{center}

\begin{remark}
For every $m, n \geq 3$, $p,q \geq 1$, if $\chi_s(C_m \times C_n) \leq k$ then $\chi_s(C_{pm} \times C_{qn}) \leq k$.
\end{remark}
Given a $k$-star coloring of $C_m \times C_n$, we can obtain a $k$-star coloring of $C_{pm} \times C_{qn}$ by repeating the coloring pattern $p$ times vertically and $q$ times horizontally. For example, a $5$-star coloring of $C_6 \times C_8$ can be obtained from a $5$-star coloring of $C_3 \times C_4$ by repeating the pattern two times vertically and two times horizontally as shown in Fig.~\ref{fig-suitable}.

\begin{figure}[ht]
 \centering
\begin{tikzpicture}[scale=.02, transform shape]
\node [draw, shape=circle] (v1) at (-300,-20) {};
\node [draw, shape=circle] (v2) at (-220,-20) {};
\node [draw, shape=circle] (v3) at (-300,-80) {};
\node [draw, shape=circle] (v4) at (-220,-80) {};

\node [draw, shape=circle] (v5) at (-140,-20) {};
\node [draw, shape=circle] (v6) at (-140,-80) {};

\node [draw, shape=circle] (v7) at (-300,-80) {};
\node [draw, shape=circle] (v8) at (-220,-80) {};
\node [draw, shape=circle] (v9) at (-300,-140) {};
\node [draw, shape=circle] (v10) at (-220,-140) {};

\node [draw, shape=circle] (v12) at (-140,-140) {};

\node [scale=50] at (-290,-30) {1};
\node [scale=50] at (-290,-50) {3};
\node [scale=50] at (-290,-70) {5};

\node [scale=50] at (-270,-30) {1};
\node [scale=50] at (-270,-50) {2};
\node [scale=50] at (-270,-70) {4};

\node [scale=50] at (-250,-30) {1};
\node [scale=50] at (-250,-50) {2};
\node [scale=50] at (-250,-70) {4};

\node [scale=50] at (-230,-30) {1};
\node [scale=50] at (-230,-50) {3};
\node [scale=50] at (-230,-70) {5};

\node [scale=50] at (-210,-30) {1};
\node [scale=50] at (-210,-50) {3};
\node [scale=50] at (-210,-70) {5};

\node [scale=50] at (-190,-30) {1};
\node [scale=50] at (-190,-50) {2};
\node [scale=50] at (-190,-70) {4};

\node [scale=50] at (-170,-30) {1};
\node [scale=50] at (-170,-50) {2};
\node [scale=50] at (-170,-70) {4};

\node [scale=50] at (-150,-30) {1};
\node [scale=50] at (-150,-50) {3};
\node [scale=50] at (-150,-70) {5};

\node [scale=50] at (-290,-90) {1};
\node [scale=50] at (-290,-110) {3};
\node [scale=50] at (-290,-130) {5};

\node [scale=50] at (-270,-90) {1};
\node [scale=50] at (-270,-110) {2};
\node [scale=50] at (-270,-130) {4};

\node [scale=50] at (-250,-90) {1};
\node [scale=50] at (-250,-110) {2};
\node [scale=50] at (-250,-130) {4};

\node [scale=50] at (-230,-90) {1};
\node [scale=50] at (-230,-110) {3};
\node [scale=50] at (-230,-130) {5};

\node [scale=50] at (-210,-90) {1};
\node [scale=50] at (-210,-110) {3};
\node [scale=50] at (-210,-130) {5};

\node [scale=50] at (-190,-90) {1};
\node [scale=50] at (-190,-110) {2};
\node [scale=50] at (-190,-130) {4};

\node [scale=50] at (-170,-90) {1};
\node [scale=50] at (-170,-110) {2};
\node [scale=50] at (-170,-130) {4};

\node [scale=50] at (-150,-90) {1};
\node [scale=50] at (-150,-110) {3};
\node [scale=50] at (-150,-130) {5};

\draw (v2)--(v5)--(v6)--(v4);
\draw (v1)--(v2)--(v4)--(v3)--(v1);
\draw (v8)--(v10)--(v9)--(v3);
\draw (v6)--(v12)--(v10);

\end{tikzpicture}
  \caption{~A $5$-star coloring of $C_6 \times C_8$ obtained using four copies of a coloring of $C_3 \times C_4$}\label{fig-suitable}
\end{figure}

\section{Tensor Product of Two Graphs}\label{S:tensor}
In this section, we give an upper bound on the star chromatic number of the tensor product of two arbitrary graphs. Next, we give exact values of the star chromatic number of tensor product of (a) two paths, (b) two cycles, and (c) a cycle and a path.

Fertin et al.~\cite{fertin2004star}  showed that $\chi_s(G \square H) \leq \chi_s(G) \chi_s(H)$. It is interesting to know an upper bound for the star chromatic number of tensor product of graphs. We observe that $\chi_s(G \times H)$ can be arbitrarily large even if $\chi_s(G)$ and $\chi_s(H)$ are constant.  For example, if $G=K_{1,n_1}$ and $H=K_{1,n_2}$, then $\chi_s(G)=\chi_s(H)=2$. 
Since $G \times H$ contains $K_{{(n_1-1),}{(n_2-1)}}$ as a subgraph, $\chi_s(G \times H) \geq \chi_s(K_{{(n_1-1),}{(n_2-1)}})=\min \{n_1-1,n_2-1\}+1$. 

In the following theorem we give an upper bound for the star chromatic number of tensor product of two arbitrary graphs.
\begin{theorem}
Let $G$ and $H$ be two connected graphs having $n_1$ and $n_2$ vertices respectively. Then we have 
 $\chi_s(G \times H) \leq \min \{n_1 \chi_s(H), n_2 \chi_s(G)\}$. 
\end{theorem}
\begin{proof}
  Let $V(G)  = \{u_1,u_2,\ldots,u_{n_1}\}$, $V(H) = \{v_1,v_2,\ldots,v_{n_2}\}$ and $V(G\times H) = \{(u_i,v_j) | i\in[n_1],j\in [n_2]\}$.
    Suppose that $\chi_s(G) = k_1$ and $\chi_s(H) = k_2$ and let $f_G:V(G) \rightarrow [k_1]$ and $f_H:V(H) \rightarrow [k_2]$ are  star colorings of $G$ and $H$ respectively. Without loss of generality, assume that $n_1 k_2 < n_2k_1$. Define $g:V(G\times H) \rightarrow [n_1k_2]$ such that $g((u_i,v_j)) = (i,f_H(v_j))$. Clearly $g$ uses $n_1k_2$ colors. 
    
    Consider any two adjacent vertices $(u_i,-)$ and $(u_j,-)$. 
    We have $g((u_i,-))=(i,-) \neq (j,-) =g((u_j,-))$. Therefore $g$ is a proper coloring of $G \times H$.
    
    Consider a path $P$ of length three having the vertices $(u_{i_1}, v_{j_1})$, $(u_{i_2}, v_{j_2})$, $(u_{i_3}, v_{j_3})$ and $(u_{i_4}, v_{j_4})$. If $u_{i_1}=u_{i_3}$ and $u_{i_2}=u_{i_4}$ then $v_{j_1}, v_{j_2}, v_{j_3},v_{j_4}$ forms a $P_4$ in the graph $H$, hence $P$ is colored with at least three distinct colors. 
    If either $u_{i_1} \neq u_{i_3}$ or $u_{i_2} \neq u_{i_4}$ then the set $\{u_{i_1}, u_{i_2}, u_{i_3}, u_{i_4}\}$ contains at least three distinct vertices of $G$, hence $P$ is colored with at least three distinct colors. Therefore, $g$ is a star coloring of $G \times H$.  \qed
\end{proof}

\subsection{Tensor product of two paths}\label{sec-PP}
In this subsection, we study the star coloring of the tensor product of two paths.
\begin{theorem}
For every pair of integers $m$ and $n$, where $2 \leq m \leq n$, we have 

$$
		\chi_s(P_m \times P_n) = 
			\begin{cases} 
				2 &\text{if $m =2$ and  $n \in \{2,3\}$}; \\
				3 &\text{if $m = 2$ and $n \geq 4$}; \\
                3 &\text{if $m = 3$ and $n \geq 3$}; \\
                4 &\text{if $m = 4$ and $n \geq 4$}; \\
                4 &\text{if $m = 5$ and $n \geq 5$}; \\
                4 &\text{if $m = 6$ and $n \in \{6,7\}$}; \\
                5 &\text{if $m = 6$ and $n \geq 8$}; \\
				5 &\text{if $m \geq 7$ and $ n \geq 7$.} \\
			\end{cases}
	$$	

\end{theorem}
\begin{proof}
\textbf{Case 1.} $m=2, n \in \{2,3\}$\\
The graphs $P_2 \times P_2$ and $P_2 \times P_3$ are disjoint union of two $P_2$'s and two $P_3$'s respectively. Hence  $\chi_s(P_2 \times P_2)= \chi_s(P_2 \times P_3)=2$.

\textbf{Case 2.} $m=2, n \geq 4$\\
 The graph $P_2 \times P_n$ is a disjoint union of two $P_n$'s. Hence $\chi_s(P_2\times P_n)=3$, for $n \geq 4$.

 \textbf{Case 3(a).} $m=3, n =3$\\
 The graph $P_3 \times P_3$ is a disjoint union of two components $C_4$ and $K_{1,4}$. As $\chi_s(C_4)=3$ and $\chi_s(K_{1,4})=2$, we have $\chi_s(P_3\times P_3)=3$.

 \textbf{Case 3(b).} $m=3, n \geq 4$\\
 The  graph $P_3 \times P_n$, for $n \geq 4$ contains two connected components. Both components contain $C_4$ as a subgraph, hence from Lemma \ref{lem-subgraph} and \ref{lem-c}, we have $\chi_s(P_3\times P_n)\geq \chi_s(C_4) = 3$ for $n \geq 4$. Also, we have shown a $3$-star coloring of $P_3 \times P_n$ in the Fig.~\ref{PP}, thus $\chi_s(P_3\times P_n)\leq 3$. Altogether, we have $\chi_s(P_3\times P_n)=3$, for $n \geq 4$. 
 
 \textbf{Case 4.} $m =4, n \geq 4$, \\
 For $n \geq 4$, the graph $P_4 \times P_n$ contains $P_2 \square P_3$ as a subgraph, hence from Lemma \ref{lem-subgraph} and \ref{lem-1}, we have $\chi_s(P_4 \times P_n) \geq \chi_s(P_2 \square P_3)=4$. 
 A $4$-star coloring of $P_4 \times P_n$, for $n \geq 4$ is shown in the Fig.~\ref{PP}. Hence $\chi_s(P_4 \times P_n) = 4$.

 \textbf{Case 5.} $m =5, n \geq 5$, \\
 For $n \geq 5$, the graph $P_5 \times P_n$ contains $P_2 \square P_3$ as a subgraph, hence from Lemma \ref{lem-subgraph} and \ref{lem-1}, we have $\chi_s(P_5 \times P_n) \geq \chi_s(P_2 \square P_3)=4$. 
 A $4$-star coloring of $P_5 \times P_n$, for $n \geq 5$ is shown in the Fig.~\ref{PP}. Hence $\chi_s(P_5 \times P_n) = 4$.

 \textbf{Case 6.} $m =6, n \in \{6,7\}$, \\
 For $n \in \{6,7\}$, the graph $P_6 \times P_n$ contains $P_2 \square P_3$ as a subgraph, hence from Lemma \ref{lem-subgraph} and \ref{lem-1}, we have $\chi_s(P_6 \times P_n) \geq \chi_s(P_2 \square P_3)=4$. 
 A $4$-star coloring of $P_6 \times P_n$, for  $n \in \{6,7\}$ is shown in  Fig.~\ref{PP}. Hence $\chi_s(P_6 \times P_n) = 4$ for $n \in \{6,7\}$.

 \textbf{Case 7.} $m =6, n \geq 8$, \\
 For $n \geq 8$, the graph $P_6 \times P_n$ contains $Z$ as a subgraph, hence  from Lemma \ref{lem-subgraph}, we have $\chi_s(P_6 \times P_n) \geq \chi_s(Z)=5$. 
 A $5$-star coloring of $P_6 \times P_n$, for $n \geq 8$ is shown in the Fig.~\ref{PP}. Hence $\chi_s(P_6 \times P_n) = 5$ for $n \geq 8$.

 \textbf{Case 8.} $m \geq 7, n\geq 7$\\
The graph $P_m\times P_n$ contains $P_4 \square P_4$ as a subgraph, hence from Lemma \ref{lem-subgraph} and \ref{lem-1}, we have $\chi_s(P_m\times P_n)\geq \chi_s(P_4\square P_4)= 5$. A $5$-star coloring of $P_m \times P_n$, for $m \geq 7, n \geq 7$ is shown in Fig.~\ref{PP}. Hence $\chi_s(P_m \times P_n) = 5$ for $m \geq 7, n \geq 7$.  \qed
\end{proof}

\begin{figure}[ht]
 \centering
\begin{tikzpicture}[scale=.02, transform shape]
\node [draw, shape=circle] (v1) at (-300,0) {};
\node [draw, shape=circle] (v2) at (-150,0) {};
\node [draw, shape=circle] (v3) at (-300,-60) {};
\node [draw, shape=circle] (v4) at (-150,-60) {};

\node [scale=50] at (-290,-10) {1};
\node [scale=50] at (-290,-30) {1};
\node [scale=50] at (-290,-50) {1};

\node [scale=50] at (-270,-10) {2};
\node [scale=50] at (-270,-30) {2};
\node [scale=50] at (-270,-50) {2};

\node [scale=50] at (-250,-10) {3};
\node [scale=50] at (-250,-30) {3};
\node [scale=50] at (-250,-50) {3};

\node [scale=50] at (-230,-10) {1};
\node [scale=50] at (-230,-30) {1};
\node [scale=50] at (-230,-50) {1};

\node [scale=50] at (-210,-10) {2};
\node [scale=50] at (-210,-30) {2};
\node [scale=50] at (-210,-50) {2};

\node [scale=50] at (-190,-10) {3};
\node [scale=50] at (-190,-30) {3};
\node [scale=50] at (-190,-50) {3};

\node [scale=50] at (-170,-10) {$\ldots$};
\node [scale=50] at (-170,-30) {$\ldots$};
\node [scale=50] at (-170,-50) {$\ldots$};

\draw (v1)--(v2)--(v4)--(v3)--(v1);
\node [scale=50] at (-260,15) {$P_3 \times P_n$};


\node [draw, shape=circle] (u1) at (-140,0) {};
\node [draw, shape=circle] (u2) at (10,0) {};
\node [draw, shape=circle] (u3) at (-140,-80) {};
\node [draw, shape=circle] (u4) at (10,-80) {};

\node [scale=50] at (-130,-10) {1};
\node [scale=50] at (-130,-30) {1};
\node [scale=50] at (-130,-50) {1};
\node [scale=50] at (-130,-70) {1};

\node [scale=50] at (-110,-10) {2};
\node [scale=50] at (-110,-30) {3};
\node [scale=50] at (-110,-50) {3};
\node [scale=50] at (-110,-70) {2};

\node [scale=50] at (-90,-10) {2};
\node [scale=50] at (-90,-30) {4};
\node [scale=50] at (-90,-50) {4};
\node [scale=50] at (-90,-70) {2};

\node [scale=50] at (-70,-10) {1};
\node [scale=50] at (-70,-30) {1};
\node [scale=50] at (-70,-50) {1};
\node [scale=50] at (-70,-70) {1};

\node [scale=50] at (-50,-10) {2};
\node [scale=50] at (-50,-30) {3};
\node [scale=50] at (-50,-50) {3};
\node [scale=50] at (-50,-70) {2};

\node [scale=50] at (-30,-10) {2};
\node [scale=50] at (-30,-30) {4};
\node [scale=50] at (-30,-50) {4};
\node [scale=50] at (-30,-70) {2};

\node [scale=50] at (-10,-10) {$\ldots$};
\node [scale=50] at (-10,-30) {$\ldots$};
\node [scale=50] at (-10,-50) {$\ldots$};
\node [scale=50] at (-10,-70) {$\ldots$};

\draw (u1)--(u2)--(u4)--(u3)--(u1);
\node [scale=50] at (-90,15) {$P_4 \times P_n$};


\node [draw, shape=circle] (u1) at (20,0) {};
\node [draw, shape=circle] (u2) at (170,0) {};
\node [draw, shape=circle] (u3) at (20,-100) {};
\node [draw, shape=circle] (u4) at (170,-100) {};

\node [scale=50] at (30,-10) {4};
\node [scale=50] at (30,-30) {1};
\node [scale=50] at (30,-50) {1};
\node [scale=50] at (30,-70) {1};
\node [scale=50] at (30,-90) {4};

\node [scale=50] at (50,-10) {4};
\node [scale=50] at (50,-30) {2};
\node [scale=50] at (50,-50) {2};
\node [scale=50] at (50,-70) {2};
\node [scale=50] at (50,-90) {4};

\node [scale=50] at (70,-10) {4};
\node [scale=50] at (70,-30) {3};
\node [scale=50] at (70,-50) {3};
\node [scale=50] at (70,-70) {3};
\node [scale=50] at (70,-90) {4};

\node [scale=50] at (90,-10) {4};
\node [scale=50] at (90,-30) {1};
\node [scale=50] at (90,-50) {1};
\node [scale=50] at (90,-70) {1};
\node [scale=50] at (90,-90) {4};

\node [scale=50] at (110,-10) {4};
\node [scale=50] at (110,-30) {2};
\node [scale=50] at (110,-50) {2};
\node [scale=50] at (110,-70) {2};
\node [scale=50] at (110,-90) {4};

\node [scale=50] at (130,-10) {4};
\node [scale=50] at (130,-30) {3};
\node [scale=50] at (130,-50) {3};
\node [scale=50] at (130,-70) {3};
\node [scale=50] at (130,-90) {4};

\node [scale=50] at (150,-10) {$\ldots$};
\node [scale=50] at (150,-30) {$\ldots$};
\node [scale=50] at (150,-50) {$\ldots$};
\node [scale=50] at (150,-70) {$\ldots$};
\node [scale=50] at (150,-90) {$\ldots$};

\draw (u1)--(u2)--(u4)--(u3)--(u1);
\node [scale=50] at (110,15) {$P_5 \times P_n$};



\node [draw, shape=circle] (u1) at (-300,-260) {};
\node [draw, shape=circle] (u2) at (-180,-260) {};
\node [draw, shape=circle] (u3) at (-300,-140) {};
\node [draw, shape=circle] (u4) at (-180,-140) {};

\node [scale=50] at (-290,-150) {2};
\node [scale=50] at (-290,-170) {1};
\node [scale=50] at (-290,-190) {2};
\node [scale=50] at (-290,-210) {2};
\node [scale=50] at (-290,-230) {1};
\node [scale=50] at (-290,-250) {2};

\node [scale=50] at (-270,-150) {3};
\node [scale=50] at (-270,-170) {1};
\node [scale=50] at (-270,-190) {4};
\node [scale=50] at (-270,-210) {4};
\node [scale=50] at (-270,-230) {1};
\node [scale=50] at (-270,-250) {3};

\node [scale=50] at (-250,-150) {2};
\node [scale=50] at (-250,-170) {2};
\node [scale=50] at (-250,-190) {3};
\node [scale=50] at (-250,-210) {3};
\node [scale=50] at (-250,-230) {2};
\node [scale=50] at (-250,-250) {2};

\node [scale=50] at (-230,-150) {3};
\node [scale=50] at (-230,-170) {4};
\node [scale=50] at (-230,-190) {1};
\node [scale=50] at (-230,-210) {1};
\node [scale=50] at (-230,-230) {4};
\node [scale=50] at (-230,-250) {3};

\node [scale=50] at (-210,-150) {1};
\node [scale=50] at (-210,-170) {4};
\node [scale=50] at (-210,-190) {2};
\node [scale=50] at (-210,-210) {2};
\node [scale=50] at (-210,-230) {4};
\node [scale=50] at (-210,-250) {1};

\node [scale=50] at (-190,-150) {3};
\node [scale=50] at (-190,-170) {3};
\node [scale=50] at (-190,-190) {3};
\node [scale=50] at (-190,-210) {3};
\node [scale=50] at (-190,-230) {3};
\node [scale=50] at (-190,-250) {3};

\draw (u1)--(u2)--(u4)--(u3)--(u1);
\node [scale=50] at (-240,-125) {$P_6 \times P_{6}$};


\node [draw, shape=circle] (u1) at (-160,-260) {};
\node [draw, shape=circle] (u2) at (-20,-260) {};
\node [draw, shape=circle] (u3) at (-160,-140) {};
\node [draw, shape=circle] (u4) at (-20,-140) {};

\node [scale=50] at (-150,-150) {3};
\node [scale=50] at (-150,-170) {3};
\node [scale=50] at (-150,-190) {3};
\node [scale=50] at (-150,-210) {3};
\node [scale=50] at (-150,-230) {3};
\node [scale=50] at (-150,-250) {3};

\node [scale=50] at (-130,-150) {1};
\node [scale=50] at (-130,-170) {4};
\node [scale=50] at (-130,-190) {2};
\node [scale=50] at (-130,-210) {2};
\node [scale=50] at (-130,-230) {4};
\node [scale=50] at (-130,-250) {1};

\node [scale=50] at (-110,-150) {3};
\node [scale=50] at (-110,-170) {4};
\node [scale=50] at (-110,-190) {1};
\node [scale=50] at (-110,-210) {1};
\node [scale=50] at (-110,-230) {4};
\node [scale=50] at (-110,-250) {3};

\node [scale=50] at (-90,-150) {2};
\node [scale=50] at (-90,-170) {2};
\node [scale=50] at (-90,-190) {3};
\node [scale=50] at (-90,-210) {3};
\node [scale=50] at (-90,-230) {2};
\node [scale=50] at (-90,-250) {2};

\node [scale=50] at (-70,-150) {3};
\node [scale=50] at (-70,-170) {1};
\node [scale=50] at (-70,-190) {4};
\node [scale=50] at (-70,-210) {4};
\node [scale=50] at (-70,-230) {1};
\node [scale=50] at (-70,-250) {3};

\node [scale=50] at (-50,-150) {4};
\node [scale=50] at (-50,-170) {1};
\node [scale=50] at (-50,-190) {2};
\node [scale=50] at (-50,-210) {2};
\node [scale=50] at (-50,-230) {1};
\node [scale=50] at (-50,-250) {4};

\node [scale=50] at (-30,-150) {3};
\node [scale=50] at (-30,-170) {3};
\node [scale=50] at (-30,-190) {3};
\node [scale=50] at (-30,-210) {3};
\node [scale=50] at (-30,-230) {3};
\node [scale=50] at (-30,-250) {3};

\draw (u1)--(u2)--(u4)--(u3)--(u1);
\node [scale=50] at (-90,-125) {$P_6 \times P_{7}$};


\node [draw, shape=circle] (u1) at (-300,-360) {};
\node [draw, shape=circle] (u2) at (-90,-360) {};
\node [draw, shape=circle] (u3) at (-300,-490) {};
\node [draw, shape=circle] (u4) at (-90,-490) {};

\node [scale=50] at (-290,-370) {1};
\node [scale=50] at (-290,-390) {1};
\node [scale=50] at (-290,-410) {1};
\node [scale=50] at (-290,-430) {1};
\node [scale=50] at (-290,-450) {1};
\node [scale=50] at (-290,-470) {1};

\node [scale=50] at (-270,-370) {2};
\node [scale=50] at (-270,-390) {2};
\node [scale=50] at (-270,-410) {3};
\node [scale=50] at (-270,-430) {3};
\node [scale=50] at (-270,-450) {2};
\node [scale=50] at (-270,-470) {2};

\node [scale=50] at (-250,-370) {4};
\node [scale=50] at (-250,-390) {4};
\node [scale=50] at (-250,-410) {5};
\node [scale=50] at (-250,-430) {5};
\node [scale=50] at (-250,-450) {4};
\node [scale=50] at (-250,-470) {4};

\node [scale=50] at (-230,-370) {1};
\node [scale=50] at (-230,-390) {1};
\node [scale=50] at (-230,-410) {1};
\node [scale=50] at (-230,-430) {1};
\node [scale=50] at (-230,-450) {1};
\node [scale=50] at (-230,-470) {1};

\node [scale=50] at (-210,-370) {2};
\node [scale=50] at (-210,-390) {2};
\node [scale=50] at (-210,-410) {3};
\node [scale=50] at (-210,-430) {3};
\node [scale=50] at (-210,-450) {2};
\node [scale=50] at (-210,-470) {2};

\node [scale=50] at (-190,-370) {4};
\node [scale=50] at (-190,-390) {4};
\node [scale=50] at (-190,-410) {5};
\node [scale=50] at (-190,-430) {5};
\node [scale=50] at (-190,-450) {4};
\node [scale=50] at (-190,-470) {4};

\node [scale=50] at (-170,-370) {1};
\node [scale=50] at (-170,-390) {1};
\node [scale=50] at (-170,-410) {1};
\node [scale=50] at (-170,-430) {1};
\node [scale=50] at (-170,-450) {1};
\node [scale=50] at (-170,-470) {1};

\node [scale=50] at (-150,-370) {2};
\node [scale=50] at (-150,-390) {2};
\node [scale=50] at (-150,-410) {3};
\node [scale=50] at (-150,-430) {3};
\node [scale=50] at (-150,-450) {2};
\node [scale=50] at (-150,-470) {2};

\node [scale=50] at (-130,-370) {4};
\node [scale=50] at (-130,-390) {4};
\node [scale=50] at (-130,-410) {5};
\node [scale=50] at (-130,-430) {5};
\node [scale=50] at (-130,-450) {4};
\node [scale=50] at (-130,-470) {4};

\node [scale=50] at (-110,-370) {\ldots};
\node [scale=50] at (-110,-390) {\ldots};
\node [scale=50] at (-110,-410) {\ldots};
\node [scale=50] at (-110,-430) {\ldots};
\node [scale=50] at (-110,-450) {\ldots};
\node [scale=50] at (-110,-470) {\ldots};

\draw (u1)--(u2)--(u4)--(u3)--(u1);
\node [scale=50] at (-210,-345) {$P_6 \times P_{n}$, $n \geq 8$};


\node [draw, shape=circle] (u1) at (-60,-300) {};
\node [draw, shape=circle] (u2) at (150,-300) {};
\node [draw, shape=circle] (u3) at (-60,-490) {};
\node [draw, shape=circle] (u4) at (150,-490) {};

\node [scale=50] at (-50,-310) {1};
\node [scale=50] at (-50,-330) {1};
\node [scale=50] at (-50,-350) {1};
\node [scale=50] at (-50,-370) {1};
\node [scale=50] at (-50,-390) {1};
\node [scale=50] at (-50,-410) {1};
\node [scale=50] at (-50,-430) {1};
\node [scale=50] at (-50,-450) {1};
\node [scale=50] at (-50,-470) {\vdots};

\node [scale=50] at (-30,-310) {2};
\node [scale=50] at (-30,-330) {2};
\node [scale=50] at (-30,-350) {3};
\node [scale=50] at (-30,-370) {3};
\node [scale=50] at (-30,-390) {2};
\node [scale=50] at (-30,-410) {2};
\node [scale=50] at (-30,-430) {3};
\node [scale=50] at (-30,-450) {3};
\node [scale=50] at (-30,-470) {\vdots};

\node [scale=50] at (-10,-310) {4};
\node [scale=50] at (-10,-330) {4};
\node [scale=50] at (-10,-350) {5};
\node [scale=50] at (-10,-370) {5};
\node [scale=50] at (-10,-390) {4};
\node [scale=50] at (-10,-410) {4};
\node [scale=50] at (-10,-430) {5};
\node [scale=50] at (-10,-450) {5};
\node [scale=50] at (-10,-470) {\vdots};

\node [scale=50] at (10,-310) {1};
\node [scale=50] at (10,-330) {1};
\node [scale=50] at (10,-350) {1};
\node [scale=50] at (10,-370) {1};
\node [scale=50] at (10,-390) {1};
\node [scale=50] at (10,-410) {1};
\node [scale=50] at (10,-430) {1};
\node [scale=50] at (10,-450) {1};
\node [scale=50] at (10,-470) {\vdots};

\node [scale=50] at (30,-310) {2};
\node [scale=50] at (30,-330) {2};
\node [scale=50] at (30,-350) {3};
\node [scale=50] at (30,-370) {3};
\node [scale=50] at (30,-390) {2};
\node [scale=50] at (30,-410) {2};
\node [scale=50] at (30,-430) {3};
\node [scale=50] at (30,-450) {3};
\node [scale=50] at (30,-470) {\vdots};

\node [scale=50] at (50,-310) {4};
\node [scale=50] at (50,-330) {4};
\node [scale=50] at (50,-350) {5};
\node [scale=50] at (50,-370) {5};
\node [scale=50] at (50,-390) {4};
\node [scale=50] at (50,-410) {4};
\node [scale=50] at (50,-430) {5};
\node [scale=50] at (50,-450) {5};
\node [scale=50] at (50,-470) {\vdots};

\node [scale=50] at (70,-310) {1};
\node [scale=50] at (70,-330) {1};
\node [scale=50] at (70,-350) {1};
\node [scale=50] at (70,-370) {1};
\node [scale=50] at (70,-390) {1};
\node [scale=50] at (70,-410) {1};
\node [scale=50] at (70,-430) {1};
\node [scale=50] at (70,-450) {1};
\node [scale=50] at (70,-470) {\vdots};

\node [scale=50] at (90,-310) {2};
\node [scale=50] at (90,-330) {2};
\node [scale=50] at (90,-350) {3};
\node [scale=50] at (90,-370) {3};
\node [scale=50] at (90,-390) {2};
\node [scale=50] at (90,-410) {2};
\node [scale=50] at (90,-430) {3};
\node [scale=50] at (90,-450) {3};
\node [scale=50] at (90,-470) {\vdots};

\node [scale=50] at (110,-310) {4};
\node [scale=50] at (110,-330) {4};
\node [scale=50] at (110,-350) {5};
\node [scale=50] at (110,-370) {5};
\node [scale=50] at (110,-390) {4};
\node [scale=50] at (110,-410) {4};
\node [scale=50] at (110,-430) {5};
\node [scale=50] at (110,-450) {5};
\node [scale=50] at (110,-470) {\vdots};

\node [scale=50] at (130,-310) {\ldots};
\node [scale=50] at (130,-330) {\ldots};
\node [scale=50] at (130,-350) {\ldots};
\node [scale=50] at (130,-370) {\ldots};
\node [scale=50] at (130,-390) {\ldots};
\node [scale=50] at (130,-410) {\ldots};
\node [scale=50] at (130,-430) {\ldots};
\node [scale=50] at (130,-450) {\ldots};

\draw (u1)--(u2)--(u4)--(u3)--(u1);
\node [scale=50] at (50,-285) {$P_m \times P_{n}$, $m,n \geq 7$};

\end{tikzpicture}
  \caption{~Star colorings $P_m \times P_n$ for various values of $m$ and $n$.}\label{PP}
\end{figure}
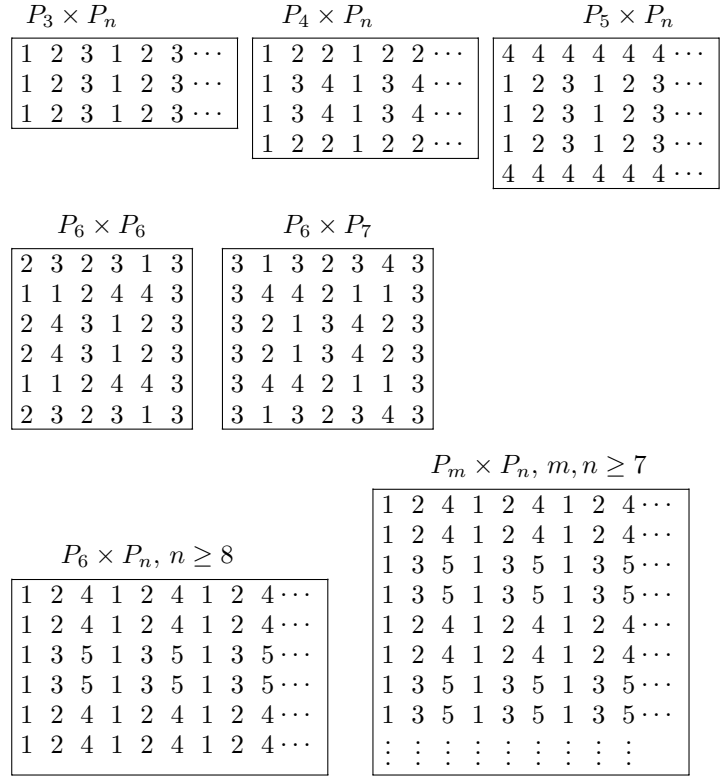


\subsection{Tensor product of two cycles}\label{sec-CC}
In this subsection, we study the star coloring  of the tensor product of two cycles. In particular, we prove the following theorem.

\begin{theorem}\label{th-CC}
    For every pair of positive integers $m$ and $n$, where $3 \leq m \leq n$, we have 
    \[ \chi_s(C_m \times C_n)=
    \begin{cases}
       6, & \text{if $m=3$, $n\in \{3,5\}$};\\
       5, & \text{otherwise}.
    \end{cases}
    \]
\end{theorem}

The proof of Theorem \ref{th-CC} follows from the following lemmas. 

\begin{lemma}\label{th-atleast5}
 For every pair of positive integers $m ,n \geq 3$, we have $\chi_s(C_m \times C_n) \geq 5$.
\end{lemma}

\begin{proof}
 The graph $C_m \times C_n$ contains $P_4 \square P_4$ as a subgraph when $m, n \geq 7$. Therefore, from Lemma \ref{lem-subgraph} and \ref{lem-1}, we have $\chi_s(C_m \times C_n) \geq \chi_s(P_4 \square P_4)=5$ for $m ,n \geq 7$. Consider the case when the minimum of $m$ and $n$ is at most $6$. Suppose $\chi_s(C_m \times C_n) \leq 4$
 and let $f$ be a $4$-star coloring of $C_m \times C_n$, then by selecting suitable copies of coloring of $C_m \times C_n$, we get a $4$-star coloring of $C_{3m} \times C_{3n}$ which is contradiction as  $C_{3m} \times C_{3n}$ contains $P_4 \square P_4$ as a subgraph, thus $\chi_s(C_{3m} \times C_{3n}) \geq 5$.  \qed
\end{proof}

\begin{lemma}~\label{thm-sylvester}\cite{sylvester1882subvariants}
Let $m$ and $n$ be two positive integers which are relatively prime. Then for every integer $k\geq(n-1)(m-1)$, there exist non-negative integers $\alpha$ and $\beta$ such that $k = \alpha n + \beta m$.
\end{lemma}

\begin{lemma}
 $\chi_s(C_3 \times C_3)=6$ and  $\chi_s(C_3 \times C_5)=6$.
\end{lemma}
\begin{proof}
 By performing a tedious case-by-case analysis we found that  $\chi_s(C_3 \times C_3)=6$ and  $\chi_s(C_3 \times C_5)=6$. Formal proof is omitted as it requires an extensive case analysis and its contribution to the theory would be minimal. \qed
\end{proof}

\begin{lemma}\label{th-c3cn}
 For every positive integer $n \geq 4$ and $n \neq 5$, we have $\chi_s(C_3 \times C_n)=5$.
\end{lemma}
\begin{proof}
 By Lemma~\ref{thm-sylvester}, every positive integer greater than or equal to $18$ can be expressed as an integer linear combination of $4$ and $7$.
 As first three columns in colorings of $C_3 \times C_4$ and $C_3 \times C_7$ are identical (see Fig.~\ref{fig-C3xCn}), by selecting suitable copies of colorings of $C_3 \times C_4$ and $C_3 \times C_7$, we can obtain a $5$-star coloring of $C_3 \times C_n$ for $n \geq 18$.  Observe that every integer $n$, $n \in \{4,6,7 \ldots, 17\}\setminus \{6,9,10,13,17\}$ is an integer linear combination of $4$ and $7$. Therefore, $5$-star coloring of $C_3 \times C_n$ can be obtained by selecting suitable copies of colorings of $C_3 \times C_4$ and $C_3 \times C_7$. $5$-star colorings of $C_3 \times C_{6}$, $C_3 \times C_{9}$ and $C_3 \times C_{10}$ are shown in the Fig.~\ref{fig-C3xCn}. Since the colors of the first three columns of $5$-star coloring of $C_3 \times C_4$ and $C_3 \times C_9$ are identical, we can obtain $5$-star colorings of $C_3 \times C_{13}$ and $C_3 \times C_{17}$ by selecting suitable copies of colorings of $C_3 \times C_4$ and $C_3 \times C_9$. Thus, by considering Fig.~\ref{fig-C3xCn} and using Lemma \ref{th-atleast5}, the proof is complete. \qed
\end{proof}

\begin{figure}[ht]
 \centering
\begin{tikzpicture}[scale=.02, transform shape]
\node [draw, shape=circle] (v1) at (-300,-20) {};
\node [draw, shape=circle] (v2) at (-220,-20) {};
\node [draw, shape=circle] (v3) at (-300,-80) {};
\node [draw, shape=circle] (v4) at (-220,-80) {};

\node [scale=50] at (-290,-30) {1};
\node [scale=50] at (-290,-50) {3};
\node [scale=50] at (-290,-70) {5};

\node [scale=50] at (-270,-30) {1};
\node [scale=50] at (-270,-50) {2};
\node [scale=50] at (-270,-70) {4};

\node [scale=50] at (-250,-30) {1};
\node [scale=50] at (-250,-50) {2};
\node [scale=50] at (-250,-70) {4};

\node [scale=50] at (-230,-30) {1};
\node [scale=50] at (-230,-50) {3};
\node [scale=50] at (-230,-70) {5};

\draw (v1)--(v2)--(v4)--(v3)--(v1);
\node [scale=50] at (-260,-5) {$C_3 \times C_4$};


\node [draw, shape=circle] (u1) at (-190,-20) {};
\node [draw, shape=circle] (u2) at (-70,-20) {};
\node [draw, shape=circle] (u3) at (-190,-80) {};
\node [draw, shape=circle] (u4) at (-70,-80) {};

\node [scale=50] at (-180,-30) {1};
\node [scale=50] at (-180,-50) {3};
\node [scale=50] at (-180,-70) {5};

\node [scale=50] at (-160,-30) {1};
\node [scale=50] at (-160,-50) {2};
\node [scale=50] at (-160,-70) {5};

\node [scale=50] at (-140,-30) {1};
\node [scale=50] at (-140,-50) {4};
\node [scale=50] at (-140,-70) {3};

\node [scale=50] at (-120,-30) {2};
\node [scale=50] at (-120,-50) {4};
\node [scale=50] at (-120,-70) {3};

\node [scale=50] at (-100,-30) {2};
\node [scale=50] at (-100,-50) {1};
\node [scale=50] at (-100,-70) {5};

\node [scale=50] at (-80,-30) {2};
\node [scale=50] at (-80,-50) {3};
\node [scale=50] at (-80,-70) {4};

\draw (u1)--(u2)--(u4)--(u3)--(u1);
\node [scale=50] at (-130,-5) {$C_3 \times C_6$};


\node [draw, shape=circle] (v1) at (-40,-20) {};
\node [draw, shape=circle] (v2) at (100,-20) {};
\node [draw, shape=circle] (v3) at (-40,-80) {};
\node [draw, shape=circle] (v4) at (100,-80) {};

\node [scale=50] at (-30,-30) {1};
\node [scale=50] at (-30,-50) {3};
\node [scale=50] at (-30,-70) {5};

\node [scale=50] at (-10,-30) {1};
\node [scale=50] at (-10,-50) {2};
\node [scale=50] at (-10,-70) {4};

\node [scale=50] at (10,-30) {1};
\node [scale=50] at (10,-50) {2};
\node [scale=50] at (10,-70) {4};

\node [scale=50] at (30,-30) {1};
\node [scale=50] at (30,-50) {3};
\node [scale=50] at (30,-70) {5};

\node [scale=50] at (50,-30) {1};
\node [scale=50] at (50,-50) {3};
\node [scale=50] at (50,-70) {5};

\node [scale=50] at (70,-30) {2};
\node [scale=50] at (70,-50) {4};
\node [scale=50] at (70,-70) {4};

\node [scale=50] at (90,-30) {2};
\node [scale=50] at (90,-50) {3};
\node [scale=50] at (90,-70) {5};

\draw (v1)--(v2)--(v4)--(v3)--(v1);
\node [scale=50] at (20,-5) {$C_3 \times C_7$};


\node [draw, shape=circle] (u1) at (130,-20) {};
\node [draw, shape=circle] (u2) at (310,-20) {};
\node [draw, shape=circle] (u3) at (130,-80) {};
\node [draw, shape=circle] (u4) at (310,-80) {};

\node [scale=50] at (140,-30) {1};
\node [scale=50] at (140,-50) {3};
\node [scale=50] at (140,-70) {5};

\node [scale=50] at (160,-30) {1};
\node [scale=50] at (160,-50) {2};
\node [scale=50] at (160,-70) {4};

\node [scale=50] at (180,-30) {1};
\node [scale=50] at (180,-50) {2};
\node [scale=50] at (180,-70) {4};

\node [scale=50] at (200,-30) {1};
\node [scale=50] at (200,-50) {3};
\node [scale=50] at (200,-70) {5};

\node [scale=50] at (220,-30) {2};
\node [scale=50] at (220,-50) {3};
\node [scale=50] at (220,-70) {5};

\node [scale=50] at (240,-30) {1};
\node [scale=50] at (240,-50) {3};
\node [scale=50] at (240,-70) {4};

\node [scale=50] at (260,-30) {1};
\node [scale=50] at (260,-50) {2};
\node [scale=50] at (260,-70) {4};

\node [scale=50] at (280,-30) {1};
\node [scale=50] at (280,-50) {2};
\node [scale=50] at (280,-70) {5};

\node [scale=50] at (300,-30) {4};
\node [scale=50] at (300,-50) {3};
\node [scale=50] at (300,-70) {5};

\draw (u1)--(u2)--(u4)--(u3)--(u1);
\node [scale=50] at (220,-5) {$C_3 \times C_9$};


\node [draw, shape=circle] (u1) at (-110,-120) {};
\node [draw, shape=circle] (u2) at (90,-120) {};
\node [draw, shape=circle] (u3) at (-110,-180) {};
\node [draw, shape=circle] (u4) at (90,-180) {};

\node [scale=50] at (-100,-130) {1};
\node [scale=50] at (-100,-150) {4};
\node [scale=50] at (-100,-170) {5};

\node [scale=50] at (-80,-130) {1};
\node [scale=50] at (-80,-150) {4};
\node [scale=50] at (-80,-170) {5};

\node [scale=50] at (-60,-130) {1};
\node [scale=50] at (-60,-150) {2};
\node [scale=50] at (-60,-170) {3};

\node [scale=50] at (-40,-130) {1};
\node [scale=50] at (-40,-150) {2};
\node [scale=50] at (-40,-170) {3};

\node [scale=50] at (-20,-130) {1};
\node [scale=50] at (-20,-150) {4};
\node [scale=50] at (-20,-170) {5};

\node [scale=50] at (0,-130) {1};
\node [scale=50] at (0,-150) {4};
\node [scale=50] at (0,-170) {5};

\node [scale=50] at (20,-130) {2};
\node [scale=50] at (20,-150) {3};
\node [scale=50] at (20,-170) {3};

\node [scale=50] at (40,-130) {2};
\node [scale=50] at (40,-150) {5};
\node [scale=50] at (40,-170) {4};

\node [scale=50] at (60,-130) {3};
\node [scale=50] at (60,-150) {5};
\node [scale=50] at (60,-170) {4};

\node [scale=50] at (80,-130) {3};
\node [scale=50] at (80,-150) {2};
\node [scale=50] at (80,-170) {2};

\draw (u1)--(u2)--(u4)--(u3)--(u1);
\node [scale=50] at (-10,-105) {$C_3 \times C_{10}$};

\end{tikzpicture}
  \caption{~$5$-star colorings of $C_3 \times C_n$, $n \in \{4,6,7,9,10\}$ }\label{fig-C3xCn}
\end{figure}


\begin{lemma}\label{th-c4cn}
 For every positive integer $n \geq 4$, we have $\chi_s(C_4 \times C_n)=5$.
\end{lemma}
\begin{proof}
 By Lemma~\ref{thm-sylvester}, every positive integer greater than or equal to $12$ can be expressed as an integer linear combination of $4$ and $5$.
 As first three columns of $C_4 \times C_4$ and $C_4 \times C_5$ are identical (see Fig.~\ref{C4}), by selecting suitable copies of colorings of $C_4 \times C_4$ and $C_4 \times C_5$, we can obtain a $5$-star coloring of $C_4 \times C_n$ for $n \geq 12$. As every integer $n \in \{4,5,8,9,10\}$ can be expressed as an integer linear combination of $4$ and $5$, we get a $5$-star coloring of $C_4 \times C_n$ for $n \in \{4,5,8,9,10\}$. $5$-star colorings of $C_4 \times C_6$, $C_4 \times C_7$ and $C_4 \times C_{11}$ are given in the Fig.~\ref{C4}. Thus, by considering Fig.~\ref{C4} and using Lemma \ref{th-atleast5}, the proof is complete. \qed
\end{proof}

\begin{figure}[ht]
 \centering
\begin{tikzpicture}[scale=.02, transform shape]
\node [draw, shape=circle] (v1) at (-300,-20) {};
\node [draw, shape=circle] (v2) at (-220,-20) {};
\node [draw, shape=circle] (v3) at (-300,-100) {};
\node [draw, shape=circle] (v4) at (-220,-100) {};

\node [scale=50] at (-290,-30) {1};
\node [scale=50] at (-290,-50) {3};
\node [scale=50] at (-290,-70) {1};
\node [scale=50] at (-290,-90) {5};

\node [scale=50] at (-270,-30) {1};
\node [scale=50] at (-270,-50) {2};
\node [scale=50] at (-270,-70) {1};
\node [scale=50] at (-270,-90) {4};

\node [scale=50] at (-250,-30) {1};
\node [scale=50] at (-250,-50) {4};
\node [scale=50] at (-250,-70) {1};
\node [scale=50] at (-250,-90) {2};

\node [scale=50] at (-230,-30) {1};
\node [scale=50] at (-230,-50) {5};
\node [scale=50] at (-230,-70) {1};
\node [scale=50] at (-230,-90) {3};

\draw (v1)--(v2)--(v4)--(v3)--(v1);
\node [scale=50] at (-260,-5) {$C_4 \times C_4$};


\node [draw, shape=circle] (u1) at (-190,-20) {};
\node [draw, shape=circle] (u2) at (-90,-20) {};
\node [draw, shape=circle] (u3) at (-190,-100) {};
\node [draw, shape=circle] (u4) at (-90,-100) {};

\node [scale=50] at (-180,-30) {1};
\node [scale=50] at (-180,-50) {3};
\node [scale=50] at (-180,-70) {1};
\node [scale=50] at (-180,-90) {5};

\node [scale=50] at (-160,-30) {1};
\node [scale=50] at (-160,-50) {2};
\node [scale=50] at (-160,-70) {1};
\node [scale=50] at (-160,-90) {4};

\node [scale=50] at (-140,-30) {1};
\node [scale=50] at (-140,-50) {4};
\node [scale=50] at (-140,-70) {1};
\node [scale=50] at (-140,-90) {2};

\node [scale=50] at (-120,-30) {1};
\node [scale=50] at (-120,-50) {5};
\node [scale=50] at (-120,-70) {1};
\node [scale=50] at (-120,-90) {3};

\node [scale=50] at (-100,-30) {2};
\node [scale=50] at (-100,-50) {5};
\node [scale=50] at (-100,-70) {4};
\node [scale=50] at (-100,-90) {3};

\draw (u1)--(u2)--(u4)--(u3)--(u1);
\node [scale=50] at (-140,-5) {$C_4 \times C_5$};


\node [draw, shape=circle] (v1) at (-40,-20) {};
\node [draw, shape=circle] (v2) at (80,-20) {};
\node [draw, shape=circle] (v3) at (-40,-100) {};
\node [draw, shape=circle] (v4) at (80,-100) {};

\node [scale=50] at (-30,-30) {1};
\node [scale=50] at (-30,-50) {3};
\node [scale=50] at (-30,-70) {1};
\node [scale=50] at (-30,-90) {5};

\node [scale=50] at (-10,-30) {1};
\node [scale=50] at (-10,-50) {2};
\node [scale=50] at (-10,-70) {1};
\node [scale=50] at (-10,-90) {4};

\node [scale=50] at (10,-30) {1};
\node [scale=50] at (10,-50) {2};
\node [scale=50] at (10,-70) {1};
\node [scale=50] at (10,-90) {4};

\node [scale=50] at (30,-30) {1};
\node [scale=50] at (30,-50) {3};
\node [scale=50] at (30,-70) {1};
\node [scale=50] at (30,-90) {5};

\node [scale=50] at (50,-30) {2};
\node [scale=50] at (50,-50) {3};
\node [scale=50] at (50,-70) {4};
\node [scale=50] at (50,-90) {5};

\node [scale=50] at (70,-30) {2};
\node [scale=50] at (70,-50) {3};
\node [scale=50] at (70,-70) {4};
\node [scale=50] at (70,-90) {5};

\draw (v1)--(v2)--(v4)--(v3)--(v1);
\node [scale=50] at (20,-5) {$C_4 \times C_6$};
\node [draw, shape=circle] (u1) at (140,-20) {};
\node [draw, shape=circle] (u2) at (280,-20) {};
\node [draw, shape=circle] (u3) at (140,-100) {};
\node [draw, shape=circle] (u4) at (280,-100) {};

\node [scale=50] at (150,-30) {1};
\node [scale=50] at (150,-50) {3};
\node [scale=50] at (150,-70) {1};
\node [scale=50] at (150,-90) {5};

\node [scale=50] at (170,-30) {1};
\node [scale=50] at (170,-50) {3};
\node [scale=50] at (170,-70) {1};
\node [scale=50] at (170,-90) {5};

\node [scale=50] at (190,-30) {1};
\node [scale=50] at (190,-50) {2};
\node [scale=50] at (190,-70) {1};
\node [scale=50] at (190,-90) {4};

\node [scale=50] at (210,-30) {1};
\node [scale=50] at (210,-50) {2};
\node [scale=50] at (210,-70) {1};
\node [scale=50] at (210,-90) {4};

\node [scale=50] at (230,-30) {1};
\node [scale=50] at (230,-50) {3};
\node [scale=50] at (230,-70) {1};
\node [scale=50] at (230,-90) {5};

\node [scale=50] at (250,-30) {1};
\node [scale=50] at (250,-50) {3};
\node [scale=50] at (250,-70) {1};
\node [scale=50] at (250,-90) {5};

\node [scale=50] at (270,-30) {2};
\node [scale=50] at (270,-50) {2};
\node [scale=50] at (270,-70) {4};
\node [scale=50] at (270,-90) {4};

\draw (u1)--(u2)--(u4)--(u3)--(u1);
\node [scale=50] at (210,-5) {$C_4 \times C_{7}$};


\node [draw, shape=circle] (u1) at (-150,-140) {};
\node [draw, shape=circle] (u2) at (70,-140) {};
\node [draw, shape=circle] (u3) at (-150,-220) {};
\node [draw, shape=circle] (u4) at (70,-220) {};

\node [scale=50] at (-140,-150) {4};
\node [scale=50] at (-140,-170) {1};
\node [scale=50] at (-140,-190) {3};
\node [scale=50] at (-140,-210) {1};

\node [scale=50] at (-120,-150) {3};
\node [scale=50] at (-120,-170) {1};
\node [scale=50] at (-120,-190) {4};
\node [scale=50] at (-120,-210) {1};

\node [scale=50] at (-100,-150) {5};
\node [scale=50] at (-100,-170) {1};
\node [scale=50] at (-100,-190) {2};
\node [scale=50] at (-100,-210) {1};

\node [scale=50] at (-80,-150) {2};
\node [scale=50] at (-80,-170) {1};
\node [scale=50] at (-80,-190) {5};
\node [scale=50] at (-80,-210) {1};

\node [scale=50] at (-60,-150) {4};
\node [scale=50] at (-60,-170) {1};
\node [scale=50] at (-60,-190) {3};
\node [scale=50] at (-60,-210) {1};

\node [scale=50] at (-40,-150) {3};
\node [scale=50] at (-40,-170) {1};
\node [scale=50] at (-40,-190) {4};
\node [scale=50] at (-40,-210) {1};

\node [scale=50] at (-20,-150) {5};
\node [scale=50] at (-20,-170) {1};
\node [scale=50] at (-20,-190) {2};
\node [scale=50] at (-20,-210) {1};

\node [scale=50] at (-0,-150) {2};
\node [scale=50] at (-0,-170) {1};
\node [scale=50] at (-0,-190) {5};
\node [scale=50] at (-0,-210) {1};

\node [scale=50] at (20,-150) {4};
\node [scale=50] at (20,-170) {1};
\node [scale=50] at (20,-190) {3};
\node [scale=50] at (20,-210) {1};

\node [scale=50] at (40,-150) {3};
\node [scale=50] at (40,-170) {1};
\node [scale=50] at (40,-190) {4};
\node [scale=50] at (40,-210) {1};

\node [scale=50] at (60,-150) {5};
\node [scale=50] at (60,-170) {2};
\node [scale=50] at (60,-190) {2};
\node [scale=50] at (60,-210) {5};

\draw (u1)--(u2)--(u4)--(u3)--(u1);
\node [scale=50] at (-40,-125) {$C_4 \times C_{11}$};

\end{tikzpicture}
  \caption{~$5$-star colorings of $C_4 \times C_n$, $n \in \{4,5,6,7,11\}$}\label{C4}
\end{figure}


\begin{lemma}\label{th-c5cn}
 For every positive integer $n \geq 5$, we have $\chi_s(C_5 \times C_n)=5$.
\end{lemma}
\begin{proof}
 By Lemma~\ref{thm-sylvester}, every positive integer greater than or equal to $12$ can be expressed as an integer linear combination of $4$ and $5$.
 As first three columns of $C_5 \times C_4$ and $C_5 \times C_5$ are identical (see Fig.~\ref{C5}), by selecting suitable copies of colorings of $C_5 \times C_4$ and $C_5 \times C_5$, we can obtain a $5$-star coloring of $C_5 \times C_n$ for $n \geq 12$. As every integer $n \in \{5,8,9,10\}$ can be expressed as an integer linear combination of $4$ and $5$, we get a $5$-star coloring of $C_5 \times C_n$ for $n \in \{5,8,9,10\}$. $5$-star colorings of $C_5 \times C_6$, $C_5 \times C_7$ and $C_5 \times C_{11}$ are given in the Fig.~\ref{C5}. Thus, by considering Fig.~\ref{C5} and using Lemma \ref{th-atleast5}, the proof is complete. \qed
\end{proof}


\begin{figure}[ht]
 \centering
\begin{tikzpicture}[scale=.02, transform shape]
\node [draw, shape=circle] (v1) at (-300,0) {};
\node [draw, shape=circle] (v2) at (-220,0) {};
\node [draw, shape=circle] (v3) at (-300,-100) {};
\node [draw, shape=circle] (v4) at (-220,-100) {};

\node [scale=50] at (-290,-10) {1};
\node [scale=50] at (-290,-30) {3};
\node [scale=50] at (-290,-50) {1};
\node [scale=50] at (-290,-70) {2};
\node [scale=50] at (-290,-90) {5};

\node [scale=50] at (-270,-10) {1};
\node [scale=50] at (-270,-30) {2};
\node [scale=50] at (-270,-50) {1};
\node [scale=50] at (-270,-70) {3};
\node [scale=50] at (-270,-90) {4};

\node [scale=50] at (-250,-10) {1};
\node [scale=50] at (-250,-30) {4};
\node [scale=50] at (-250,-50) {4};
\node [scale=50] at (-250,-70) {5};
\node [scale=50] at (-250,-90) {2};

\node [scale=50] at (-230,-10) {1};
\node [scale=50] at (-230,-30) {5};
\node [scale=50] at (-230,-50) {1};
\node [scale=50] at (-230,-70) {3};
\node [scale=50] at (-230,-90) {3};

\draw (v1)--(v2)--(v4)--(v3)--(v1);
\node [scale=50] at (-260,15) {$C_5 \times C_4$};


\node [draw, shape=circle] (u1) at (-200,0) {};
\node [draw, shape=circle] (u2) at (-100,0) {};
\node [draw, shape=circle] (u3) at (-200,-100) {};
\node [draw, shape=circle] (u4) at (-100,-100) {};

\node [scale=50] at (-190,-10) {1};
\node [scale=50] at (-190,-30) {3};
\node [scale=50] at (-190,-50) {1};
\node [scale=50] at (-190,-70) {2};
\node [scale=50] at (-190,-90) {5};

\node [scale=50] at (-170,-10) {1};
\node [scale=50] at (-170,-30) {2};
\node [scale=50] at (-170,-50) {1};
\node [scale=50] at (-170,-70) {3};
\node [scale=50] at (-170,-90) {4};

\node [scale=50] at (-150,-10) {1};
\node [scale=50] at (-150,-30) {4};
\node [scale=50] at (-150,-50) {4};
\node [scale=50] at (-150,-70) {5};
\node [scale=50] at (-150,-90) {2};

\node [scale=50] at (-130,-10) {1};
\node [scale=50] at (-130,-30) {5};
\node [scale=50] at (-130,-50) {2};
\node [scale=50] at (-130,-70) {1};
\node [scale=50] at (-130,-90) {3};

\node [scale=50] at (-110,-10) {2};
\node [scale=50] at (-110,-30) {5};
\node [scale=50] at (-110,-50) {4};
\node [scale=50] at (-110,-70) {4};
\node [scale=50] at (-110,-90) {3};

\draw (u1)--(u2)--(u4)--(u3)--(u1);
\node [scale=50] at (-140,15) {$C_5 \times C_5$};


\node [draw, shape=circle] (v1) at (-60,0) {};
\node [draw, shape=circle] (v2) at (60,0) {};
\node [draw, shape=circle] (v3) at (-60,-100) {};
\node [draw, shape=circle] (v4) at (60,-100) {};

\node [scale=50] at (-50,-10) {1};
\node [scale=50] at (-50,-30) {2};
\node [scale=50] at (-50,-50) {3};
\node [scale=50] at (-50,-70) {1};
\node [scale=50] at (-50,-90) {2};

\node [scale=50] at (-30,-10) {1};
\node [scale=50] at (-30,-30) {2};
\node [scale=50] at (-30,-50) {4};
\node [scale=50] at (-30,-70) {5};
\node [scale=50] at (-30,-90) {3};

\node [scale=50] at (-10,-10) {1};
\node [scale=50] at (-10,-30) {3};
\node [scale=50] at (-10,-50) {3};
\node [scale=50] at (-10,-70) {5};
\node [scale=50] at (-10,-90) {4};

\node [scale=50] at (10,-10) {1};
\node [scale=50] at (10,-30) {4};
\node [scale=50] at (10,-50) {2};
\node [scale=50] at (10,-70) {1};
\node [scale=50] at (10,-90) {4};

\node [scale=50] at (30,-10) {2};
\node [scale=50] at (30,-30) {5};
\node [scale=50] at (30,-50) {5};
\node [scale=50] at (30,-70) {3};
\node [scale=50] at (30,-90) {2};

\node [scale=50] at (50,-10) {3};
\node [scale=50] at (50,-30) {4};
\node [scale=50] at (50,-50) {4};
\node [scale=50] at (50,-70) {4};
\node [scale=50] at (50,-90) {5};

\draw (v1)--(v2)--(v4)--(v3)--(v1);
\node [scale=50] at (-10,15) {$C_5 \times C_6$};


\node [draw, shape=circle] (u1) at (100,0) {};
\node [draw, shape=circle] (u2) at (240,0) {};
\node [draw, shape=circle] (u3) at (100,-100) {};
\node [draw, shape=circle] (u4) at (240,-100) {};

\node [scale=50] at (110,-10) {1};
\node [scale=50] at (110,-30) {3};
\node [scale=50] at (110,-50) {1};
\node [scale=50] at (110,-70) {2};
\node [scale=50] at (110,-90) {4};

\node [scale=50] at (130,-10) {1};
\node [scale=50] at (130,-30) {4};
\node [scale=50] at (130,-50) {4};
\node [scale=50] at (130,-70) {3};
\node [scale=50] at (130,-90) {3};

\node [scale=50] at (150,-10) {1};
\node [scale=50] at (150,-30) {2};
\node [scale=50] at (150,-50) {1};
\node [scale=50] at (150,-70) {5};
\node [scale=50] at (150,-90) {5};

\node [scale=50] at (170,-10) {1};
\node [scale=50] at (170,-30) {5};
\node [scale=50] at (170,-50) {1};
\node [scale=50] at (170,-70) {2};
\node [scale=50] at (170,-90) {2};

\node [scale=50] at (190,-10) {1};
\node [scale=50] at (190,-30) {3};
\node [scale=50] at (190,-50) {3};
\node [scale=50] at (190,-70) {4};
\node [scale=50] at (190,-90) {4};

\node [scale=50] at (210,-10) {1};
\node [scale=50] at (210,-30) {4};
\node [scale=50] at (210,-50) {1};
\node [scale=50] at (210,-70) {2};
\node [scale=50] at (210,-90) {3};

\node [scale=50] at (230,-10) {2};
\node [scale=50] at (230,-30) {2};
\node [scale=50] at (230,-50) {1};
\node [scale=50] at (230,-70) {5};
\node [scale=50] at (230,-90) {5};

\draw (u1)--(u2)--(u4)--(u3)--(u1);
\node [scale=50] at (170,15) {$C_5 \times C_7$};


\node [draw, shape=circle] (u1) at (-150,-140) {};
\node [draw, shape=circle] (u2) at (70,-140) {};
\node [draw, shape=circle] (u3) at (-150,-240) {};
\node [draw, shape=circle] (u4) at (70,-240) {};

\node [scale=50] at (-140,-150) {4};
\node [scale=50] at (-140,-170) {1};
\node [scale=50] at (-140,-190) {3};
\node [scale=50] at (-140,-210) {1};
\node [scale=50] at (-140,-230) {2};

\node [scale=50] at (-120,-150) {3};
\node [scale=50] at (-120,-170) {1};
\node [scale=50] at (-120,-190) {4};
\node [scale=50] at (-120,-210) {4};
\node [scale=50] at (-120,-230) {3};

\node [scale=50] at (-100,-150) {5};
\node [scale=50] at (-100,-170) {1};
\node [scale=50] at (-100,-190) {2};
\node [scale=50] at (-100,-210) {1};
\node [scale=50] at (-100,-230) {5};

\node [scale=50] at (-80,-150) {2};
\node [scale=50] at (-80,-170) {1};
\node [scale=50] at (-80,-190) {5};
\node [scale=50] at (-80,-210) {1};
\node [scale=50] at (-80,-230) {2};

\node [scale=50] at (-60,-150) {4};
\node [scale=50] at (-60,-170) {1};
\node [scale=50] at (-60,-190) {3};
\node [scale=50] at (-60,-210) {3};
\node [scale=50] at (-60,-230) {4};

\node [scale=50] at (-40,-150) {3};
\node [scale=50] at (-40,-170) {1};
\node [scale=50] at (-40,-190) {4};
\node [scale=50] at (-40,-210) {1};
\node [scale=50] at (-40,-230) {2};

\node [scale=50] at (-20,-150) {5};
\node [scale=50] at (-20,-170) {1};
\node [scale=50] at (-20,-190) {2};
\node [scale=50] at (-20,-210) {1};
\node [scale=50] at (-20,-230) {5};

\node [scale=50] at (-0,-150) {2};
\node [scale=50] at (-0,-170) {1};
\node [scale=50] at (-0,-190) {5};
\node [scale=50] at (-0,-210) {1};
\node [scale=50] at (0,-230) {2};

\node [scale=50] at (20,-150) {4};
\node [scale=50] at (20,-170) {1};
\node [scale=50] at (20,-190) {3};
\node [scale=50] at (20,-210) {3};
\node [scale=50] at (20,-230) {4};

\node [scale=50] at (40,-150) {3};
\node [scale=50] at (40,-170) {1};
\node [scale=50] at (40,-190) {4};
\node [scale=50] at (40,-210) {1};
\node [scale=50] at (40,-230) {2};

\node [scale=50] at (60,-150) {5};
\node [scale=50] at (60,-170) {2};
\node [scale=50] at (60,-190) {2};
\node [scale=50] at (60,-210) {1};
\node [scale=50] at (60,-230) {5};

\draw (u1)--(u2)--(u4)--(u3)--(u1);
\node [scale=50] at (-40,-125) {$C_5 \times C_{11}$};

\end{tikzpicture}
  \caption{~$5$-star colorings of $C_5 \times C_n$, $n \in \{4,5,6,7,11\}$}\label{C5}
\end{figure}


\begin{lemma}\label{C7Cn}
 For every positive integer $n \geq 7$, we have $\chi_s(C_7 \times C_n)=5$.
\end{lemma}

\begin{proof}
 By Lemma~\ref{thm-sylvester}, every positive integer greater than or equal to $12$ can be expressed as an integer linear combination of $4$ and $5$.
 As first three columns of $C_7 \times C_4$ and $C_7 \times C_5$ are identical (see Fig.~\ref{C7}), by selecting suitable copies of colorings of $C_7 \times C_4$ and $C_7 \times C_5$, we can obtain a $5$-star coloring of $C_7 \times C_n$ for $n \geq 12$. As every integer $n \in \{8,9,10\}$ can be expressed as an integer linear combination of $4$ and $5$, we get a $5$-star coloring of $C_7 \times C_n$ for $n \in \{8,9,10\}$. $5$-star colorings of $C_7 \times C_7$ and $C_7 \times C_{11}$ are given in the Fig.~\ref{C7}. Thus, by considering Fig.~\ref{C7} and using Lemma \ref{th-atleast5}, the proof is complete. \qed
 
\end{proof}

\begin{figure}[ht]
 \centering
\begin{tikzpicture}[scale=.02, transform shape]
\node [draw, shape=circle] (v1) at (-560,-20) {};
\node [draw, shape=circle] (v2) at (-480,-20) {};
\node [draw, shape=circle] (v3) at (-560,-160) {};
\node [draw, shape=circle] (v4) at (-480,-160) {};

\node [scale=50] at (-550,-30) {4};
\node [scale=50] at (-550,-50) {3};
\node [scale=50] at (-550,-70) {5};
\node [scale=50] at (-550,-90) {2};
\node [scale=50] at (-550,-110) {4};
\node [scale=50] at (-550,-130) {3};
\node [scale=50] at (-550,-150) {5};

\node [scale=50] at (-530,-30) {1};
\node [scale=50] at (-530,-50) {1};
\node [scale=50] at (-530,-70) {1};
\node [scale=50] at (-530,-90) {1};
\node [scale=50] at (-530,-110) {1};
\node [scale=50] at (-530,-130) {1};
\node [scale=50] at (-530,-150) {2};

\node [scale=50] at (-510,-30) {3};
\node [scale=50] at (-510,-50) {4};
\node [scale=50] at (-510,-70) {2};
\node [scale=50] at (-510,-90) {5};
\node [scale=50] at (-510,-110) {3};
\node [scale=50] at (-510,-130) {4};
\node [scale=50] at (-510,-150) {2};

\node [scale=50] at (-490,-30) {1};
\node [scale=50] at (-490,-50) {1};
\node [scale=50] at (-490,-70) {1};
\node [scale=50] at (-490,-90) {1};
\node [scale=50] at (-490,-110) {1};
\node [scale=50] at (-490,-130) {1};
\node [scale=50] at (-490,-150) {5};

\draw (v1)--(v2)--(v4)--(v3)--(v1);
\node [scale=50] at (-520,-5) {$C_7 \times C_4$};


\node [draw, shape=circle] (u1) at (-460,-20) {};
\node [draw, shape=circle] (u2) at (-360,-20) {};
\node [draw, shape=circle] (u3) at (-460,-160) {};
\node [draw, shape=circle] (u4) at (-360,-160) {};

\node [scale=50] at (-450,-30) {4};
\node [scale=50] at (-450,-50) {3};
\node [scale=50] at (-450,-70) {5};
\node [scale=50] at (-450,-90) {2};
\node [scale=50] at (-450,-110) {4};
\node [scale=50] at (-450,-130) {3};
\node [scale=50] at (-450,-150) {5};

\node [scale=50] at (-430,-30) {1};
\node [scale=50] at (-430,-50) {1};
\node [scale=50] at (-430,-70) {1};
\node [scale=50] at (-430,-90) {1};
\node [scale=50] at (-430,-110) {1};
\node [scale=50] at (-430,-130) {1};
\node [scale=50] at (-430,-150) {2};

\node [scale=50] at (-410,-30) {3};
\node [scale=50] at (-410,-50) {4};
\node [scale=50] at (-410,-70) {2};
\node [scale=50] at (-410,-90) {5};
\node [scale=50] at (-410,-110) {3};
\node [scale=50] at (-410,-130) {4};
\node [scale=50] at (-410,-150) {2};

\node [scale=50] at (-390,-30) {1};
\node [scale=50] at (-390,-50) {4};
\node [scale=50] at (-390,-70) {1};
\node [scale=50] at (-390,-90) {1};
\node [scale=50] at (-390,-110) {3};
\node [scale=50] at (-390,-130) {1};
\node [scale=50] at (-390,-150) {1};

\node [scale=50] at (-370,-30) {2};
\node [scale=50] at (-370,-50) {3};
\node [scale=50] at (-370,-70) {5};
\node [scale=50] at (-370,-90) {2};
\node [scale=50] at (-370,-110) {4};
\node [scale=50] at (-370,-130) {2};
\node [scale=50] at (-370,-150) {5};

\draw (u1)--(u2)--(u4)--(u3)--(u1);
\node [scale=50] at (-410,-5) {$C_7 \times C_5$};


\node [draw, shape=circle] (u1) at (-340,-20) {};
\node [draw, shape=circle] (u2) at (-200,-20) {};
\node [draw, shape=circle] (u3) at (-340,-160) {};
\node [draw, shape=circle] (u4) at (-200,-160) {};

\node [scale=50] at (-330,-30) {1};
\node [scale=50] at (-330,-50) {5};
\node [scale=50] at (-330,-70) {1};
\node [scale=50] at (-330,-90) {2};
\node [scale=50] at (-330,-110) {5};
\node [scale=50] at (-330,-130) {1};
\node [scale=50] at (-330,-150) {2};

\node [scale=50] at (-310,-30) {1};
\node [scale=50] at (-310,-50) {5};
\node [scale=50] at (-310,-70) {1};
\node [scale=50] at (-310,-90) {2};
\node [scale=50] at (-310,-110) {5};
\node [scale=50] at (-310,-130) {1};
\node [scale=50] at (-310,-150) {2};

\node [scale=50] at (-290,-30) {1};
\node [scale=50] at (-290,-50) {3};
\node [scale=50] at (-290,-70) {1};
\node [scale=50] at (-290,-90) {4};
\node [scale=50] at (-290,-110) {3};
\node [scale=50] at (-290,-130) {1};
\node [scale=50] at (-290,-150) {4};

\node [scale=50] at (-270,-30) {1};
\node [scale=50] at (-270,-50) {3};
\node [scale=50] at (-270,-70) {1};
\node [scale=50] at (-270,-90) {4};
\node [scale=50] at (-270,-110) {3};
\node [scale=50] at (-270,-130) {1};
\node [scale=50] at (-270,-150) {4};

\node [scale=50] at (-250,-30) {1};
\node [scale=50] at (-250,-50) {5};
\node [scale=50] at (-250,-70) {2};
\node [scale=50] at (-250,-90) {2};
\node [scale=50] at (-250,-110) {5};
\node [scale=50] at (-250,-130) {2};
\node [scale=50] at (-250,-150) {2};

\node [scale=50] at (-230,-30) {1};
\node [scale=50] at (-230,-50) {5};
\node [scale=50] at (-230,-70) {1};
\node [scale=50] at (-230,-90) {4};
\node [scale=50] at (-230,-110) {5};
\node [scale=50] at (-230,-130) {1};
\node [scale=50] at (-230,-150) {5};

\node [scale=50] at (-210,-30) {4};
\node [scale=50] at (-210,-50) {3};
\node [scale=50] at (-210,-70) {3};
\node [scale=50] at (-210,-90) {4};
\node [scale=50] at (-210,-110) {3};
\node [scale=50] at (-210,-130) {3};
\node [scale=50] at (-210,-150) {4};

\draw (u1)--(u2)--(u4)--(u3)--(u1);
\node [scale=50] at (-270,-5) {$C_7 \times C_7$};


\node [draw, shape=circle] (u1) at (-180,-20) {};
\node [draw, shape=circle] (u2) at (40,-20) {};
\node [draw, shape=circle] (u3) at (-180,-160) {};
\node [draw, shape=circle] (u4) at (40,-160) {};

\node [scale=50] at (-170,-30) {1};
\node [scale=50] at (-170,-50) {1};
\node [scale=50] at (-170,-70) {1};
\node [scale=50] at (-170,-90) {1};
\node [scale=50] at (-170,-110) {1};
\node [scale=50] at (-170,-130) {1};
\node [scale=50] at (-170,-150) {4};

\node [scale=50] at (-150,-30) {5};
\node [scale=50] at (-150,-50) {5};
\node [scale=50] at (-150,-70) {3};
\node [scale=50] at (-150,-90) {3};
\node [scale=50] at (-150,-110) {5};
\node [scale=50] at (-150,-130) {5};
\node [scale=50] at (-150,-150) {3};

\node [scale=50] at (-130,-30) {1};
\node [scale=50] at (-130,-50) {1};
\node [scale=50] at (-130,-70) {1};
\node [scale=50] at (-130,-90) {1};
\node [scale=50] at (-130,-110) {2};
\node [scale=50] at (-130,-130) {1};
\node [scale=50] at (-130,-150) {3};

\node [scale=50] at (-110,-30) {2};
\node [scale=50] at (-110,-50) {2};
\node [scale=50] at (-110,-70) {4};
\node [scale=50] at (-110,-90) {4};
\node [scale=50] at (-110,-110) {2};
\node [scale=50] at (-110,-130) {4};
\node [scale=50] at (-110,-150) {4};

\node [scale=50] at (-90,-30) {5};
\node [scale=50] at (-90,-50) {5};
\node [scale=50] at (-90,-70) {3};
\node [scale=50] at (-90,-90) {3};
\node [scale=50] at (-90,-110) {5};
\node [scale=50] at (-90,-130) {5};
\node [scale=50] at (-90,-150) {3};

\node [scale=50] at (-70,-30) {1};
\node [scale=50] at (-70,-50) {1};
\node [scale=50] at (-70,-70) {1};
\node [scale=50] at (-70,-90) {1};
\node [scale=50] at (-70,-110) {2};
\node [scale=50] at (-70,-130) {1};
\node [scale=50] at (-70,-150) {3};

\node [scale=50] at (-50,-30) {2};
\node [scale=50] at (-50,-50) {2};
\node [scale=50] at (-50,-70) {4};
\node [scale=50] at (-50,-90) {4};
\node [scale=50] at (-50,-110) {2};
\node [scale=50] at (-50,-130) {5};
\node [scale=50] at (-50,-150) {4};

\node [scale=50] at (-30,-30) {1};
\node [scale=50] at (-30,-50) {1};
\node [scale=50] at (-30,-70) {1};
\node [scale=50] at (-30,-90) {1};
\node [scale=50] at (-30,-110) {1};
\node [scale=50] at (-30,-130) {1};
\node [scale=50] at (-30,-150) {4};

\node [scale=50] at (-10,-30) {5};
\node [scale=50] at (-10,-50) {5};
\node [scale=50] at (-10,-70) {3};
\node [scale=50] at (-10,-90) {3};
\node [scale=50] at (-10,-110) {5};
\node [scale=50] at (-10,-130) {5};
\node [scale=50] at (-10,-150) {3};

\node [scale=50] at (10,-30) {1};
\node [scale=50] at (10,-50) {1};
\node [scale=50] at (10,-70) {1};
\node [scale=50] at (10,-90) {1};
\node [scale=50] at (10,-110) {2};
\node [scale=50] at (10,-130) {1};
\node [scale=50] at (10,-150) {3};

\node [scale=50] at (30,-30) {2};
\node [scale=50] at (30,-50) {2};
\node [scale=50] at (30,-70) {4};
\node [scale=50] at (30,-90) {4};
\node [scale=50] at (30,-110) {2};
\node [scale=50] at (30,-130) {5};
\node [scale=50] at (30,-150) {4};

\draw (u1)--(u2)--(u4)--(u3)--(u1);
\node [scale=50] at (-70,-5) {$C_7 \times C_{11}$};

\end{tikzpicture}
  \caption{~$5$-star colorings of $C_7 \times C_n$, $n \in \{4,5,7,11\}$}\label{C7}
\end{figure}

\begin{lemma}\label{th-C11}
 For every positive integer $n \geq 11$, we have $\chi_s(C_{11} \times C_n)=5$.
\end{lemma}
\begin{proof}
By Lemma~\ref{thm-sylvester}, every positive integer greater than or equal to $12$ can be expressed as an integer linear combination of $4$ and $5$.
 As first three rows of $C_4 \times C_{11}$ (see Fig.~\ref{C4}) and $C_5 \times C_{11}$ (see Fig.~\ref{C5}) are identical, by selecting suitable copies of colorings of $C_4 \times C_{11}$ and $C_5 \times C_{11}$, we can obtain a $5$-star coloring of $C_{11} \times C_n$ for $n \geq 12$.
 For $n=11$ we have given a $5$-star coloring of $C_{11} \times C_{11}$ in the Fig.~\ref{C11}. Now, by using Lemma \ref{th-atleast5}, the proof is complete. \qed

\end{proof}



\begin{figure}[ht]
 \centering
\begin{tikzpicture}[scale=.02, transform shape]

\node [draw, shape=circle] (u1) at (-340,-20) {};
\node [draw, shape=circle] (u2) at (-120,-20) {};
\node [draw, shape=circle] (u3) at (-340,-240) {};
\node [draw, shape=circle] (u4) at (-120,-240) {};

\node [scale=50] at (-330,-30) {1};
\node [scale=50] at (-330,-50) {3};
\node [scale=50] at (-330,-70) {1};
\node [scale=50] at (-330,-90) {5};
\node [scale=50] at (-330,-110) {1};
\node [scale=50] at (-330,-130) {3};
\node [scale=50] at (-330,-150) {1};
\node [scale=50] at (-330,-170) {5};
\node [scale=50] at (-330,-190) {1};
\node [scale=50] at (-330,-210) {3};
\node [scale=50] at (-330,-230) {5};

\node [scale=50] at (-310,-30) {1};
\node [scale=50] at (-310,-50) {2};
\node [scale=50] at (-310,-70) {1};
\node [scale=50] at (-310,-90) {4};
\node [scale=50] at (-310,-110) {1};
\node [scale=50] at (-310,-130) {2};
\node [scale=50] at (-310,-150) {1};
\node [scale=50] at (-310,-170) {4};
\node [scale=50] at (-310,-190) {1};
\node [scale=50] at (-310,-210) {2};
\node [scale=50] at (-310,-230) {4};

\node [scale=50] at (-290,-30) {1};
\node [scale=50] at (-290,-50) {2};
\node [scale=50] at (-290,-70) {1};
\node [scale=50] at (-290,-90) {4};
\node [scale=50] at (-290,-110) {1};
\node [scale=50] at (-290,-130) {2};
\node [scale=50] at (-290,-150) {1};
\node [scale=50] at (-290,-170) {4};
\node [scale=50] at (-290,-190) {1};
\node [scale=50] at (-290,-210) {2};
\node [scale=50] at (-290,-230) {4};

\node [scale=50] at (-270,-30) {1};
\node [scale=50] at (-270,-50) {3};
\node [scale=50] at (-270,-70) {1};
\node [scale=50] at (-270,-90) {5};
\node [scale=50] at (-270,-110) {1};
\node [scale=50] at (-270,-130) {3};
\node [scale=50] at (-270,-150) {1};
\node [scale=50] at (-270,-170) {5};
\node [scale=50] at (-270,-190) {1};
\node [scale=50] at (-270,-210) {3};
\node [scale=50] at (-270,-230) {5};

\node [scale=50] at (-250,-30) {2};
\node [scale=50] at (-250,-50) {3};
\node [scale=50] at (-250,-70) {4};
\node [scale=50] at (-250,-90) {5};
\node [scale=50] at (-250,-110) {2};
\node [scale=50] at (-250,-130) {3};
\node [scale=50] at (-250,-150) {4};
\node [scale=50] at (-250,-170) {5};
\node [scale=50] at (-250,-190) {1};
\node [scale=50] at (-250,-210) {3};
\node [scale=50] at (-250,-230) {5};

\node [scale=50] at (-230,-30) {1};
\node [scale=50] at (-230,-50) {3};
\node [scale=50] at (-230,-70) {1};
\node [scale=50] at (-230,-90) {5};
\node [scale=50] at (-230,-110) {1};
\node [scale=50] at (-230,-130) {3};
\node [scale=50] at (-230,-150) {1};
\node [scale=50] at (-230,-170) {5};
\node [scale=50] at (-230,-190) {1};
\node [scale=50] at (-230,-210) {2};
\node [scale=50] at (-230,-230) {4};

\node [scale=50] at (-210,-30) {1};
\node [scale=50] at (-210,-50) {2};
\node [scale=50] at (-210,-70) {1};
\node [scale=50] at (-210,-90) {4};
\node [scale=50] at (-210,-110) {1};
\node [scale=50] at (-210,-130) {2};
\node [scale=50] at (-210,-150) {1};
\node [scale=50] at (-210,-170) {4};
\node [scale=50] at (-210,-190) {1};
\node [scale=50] at (-210,-210) {2};
\node [scale=50] at (-210,-230) {4};

\node [scale=50] at (-190,-30) {1};
\node [scale=50] at (-190,-50) {2};
\node [scale=50] at (-190,-70) {1};
\node [scale=50] at (-190,-90) {4};
\node [scale=50] at (-190,-110) {1};
\node [scale=50] at (-190,-130) {2};
\node [scale=50] at (-190,-150) {1};
\node [scale=50] at (-190,-170) {4};
\node [scale=50] at (-190,-190) {1};
\node [scale=50] at (-190,-210) {3};
\node [scale=50] at (-190,-230) {5};

\node [scale=50] at (-170,-30) {1};
\node [scale=50] at (-170,-50) {3};
\node [scale=50] at (-170,-70) {1};
\node [scale=50] at (-170,-90) {5};
\node [scale=50] at (-170,-110) {1};
\node [scale=50] at (-170,-130) {3};
\node [scale=50] at (-170,-150) {1};
\node [scale=50] at (-170,-170) {5};
\node [scale=50] at (-170,-190) {1};
\node [scale=50] at (-170,-210) {3};
\node [scale=50] at (-170,-230) {5};

\node [scale=50] at (-150,-30) {2};
\node [scale=50] at (-150,-50) {3};
\node [scale=50] at (-150,-70) {4};
\node [scale=50] at (-150,-90) {5};
\node [scale=50] at (-150,-110) {2};
\node [scale=50] at (-150,-130) {3};
\node [scale=50] at (-150,-150) {4};
\node [scale=50] at (-150,-170) {5};
\node [scale=50] at (-150,-190) {2};
\node [scale=50] at (-150,-210) {4};
\node [scale=50] at (-150,-230) {4};

\node [scale=50] at (-130,-30) {2};
\node [scale=50] at (-130,-50) {3};
\node [scale=50] at (-130,-70) {4};
\node [scale=50] at (-130,-90) {5};
\node [scale=50] at (-130,-110) {2};
\node [scale=50] at (-130,-130) {3};
\node [scale=50] at (-130,-150) {4};
\node [scale=50] at (-130,-170) {5};
\node [scale=50] at (-130,-190) {2};
\node [scale=50] at (-130,-210) {3};
\node [scale=50] at (-130,-230) {5};

\draw (u1)--(u2)--(u4)--(u3)--(u1);
\node [scale=50] at (-230,-5) {$C_{11} \times C_{11}$};

\end{tikzpicture}
  \caption{~A $5$-star coloring of $C_{11} \times C_{11}$.}\label{C11}
\end{figure}


\begin{lemma}\label{th-C6C8C9C10}
 For every positive integer $n \geq m$, where $m \in \{6,8,9,10\}$ we have $\chi_s(C_{m} \times C_n)=5$.
\end{lemma}
\begin{proof}
For every natural number $n \geq m$, where $m \in \{6,9\}$, we get  $5$-star colorings of $C_6 \times C_n$ and $C_9 \times C_n$ from the $5$-star coloring of $C_3 \times C_n$. 
We get $5$-star colorings of $C_8 \times C_n$ and $C_{10} \times C_n$ from the 5-star colorings of $C_4 \times C_n$ and $C_5 \times C_n$ respectively. Now, by using Lemma \ref{th-atleast5}, the proof is complete. \qed
     
\end{proof}

 \begin{lemma}
 For every pair of positive integers $m$ and $n$, where $12 \leq m \leq n$,  we have 
 $\chi_s(C_m \times C_n)=5$
 \end{lemma}
 \begin{proof}
 By Lemma~\ref{thm-sylvester}, every positive integer greater than or equal to $12$ can be expressed as an integer linear combination of $4$ and $5$.
 We have given $5$-star colorings of $C_4 \times C_4$, $C_4 \times C_5$ in Fig.~\ref{C4} and $C_5 \times C_4$ and $C_5 \times C_5$ in Fig.~\ref{C5} such that 
 \begin{itemize}
  \item The colors of the first three columns of $C_4 \times C_4$ and $C_4 \times C_5$ are same. 
  
  \item The colors of the first three columns of $C_5 \times C_4$ and $C_5 \times C_5$ are same. 
  
   \item The colors of the first two rows and the last row of $C_4 \times C_5$ and $C_5 \times C_5$ are the same. 
  
   \item The colors of the first two rows and the last row of $C_5 \times C_4$ and $C_4 \times C_4$ are the same.
 \end{itemize}

 By selecting suitable copies of the colorings of $C_4 \times C_4$, $C_4 \times C_5$, $C_5 \times C_4$ and $C_5 \times C_5$, we can obtain a $5$-star coloring of $C_m \times C_n$ for $12 \leq m \leq n$. For example, a $5$-star coloring of $C_{14} \times C_{19}$ can be obtained from the colorings of $C_4 \times C_4$, $C_4 \times C_5$, $C_5 \times C_4$ and $C_5 \times C_5$ as shown in Fig.~\ref{fig:block}. \qed
 \end{proof}

\begin{figure}
    \centering
    \includegraphics[trim=3cm 16.5cm 6cm 4cm, clip=true, scale=0.8]{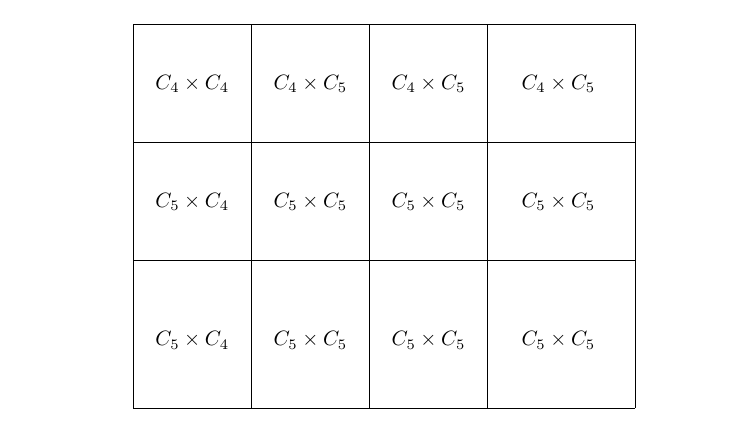}
    \caption{~A $5$-star coloring of $C_{14} \times C_{19}$ can be obtained from the colorings of $C_4 \times C_4$, $C_4 \times C_5$, $C_5 \times C_4$ and $C_5 \times C_5$.}
    \label{fig:block}
\end{figure}


 \subsection{Tensor product of a cycle and a path}\label{sec-CP}
In this subsection, we study star the coloring of the tensor product of a path and a cycle.  In particular, we prove the following theorem.

\begin{theorem}\label{th-CP}
    For every pair of integers $m \geq 3$ and $n \geq 2$, we have
    \[
    \chi_s(C_m \times P_n) = 
    \begin{cases}
        3, & \text{if $m \geq 3$, $n\in \{2,3\}$};\\
        4, & \text{if $m=3k$, $k\in \mathbb{N}$, $n \in \{4,5\}$};\\
        \leq 5, & \text{if $m\neq 3k$, $k\in \mathbb{N}$, $n \in \{4,5\}$};\\
        5, & \text{otherwise.}
    \end{cases}
    \]
\end{theorem}

The proof of Theorem \ref{th-CP} follows from the following lemmas.
\begin{lemma}
 For every integer $m$, where $m \geq 3$, we have $\chi_s(C_m \times P_2)=3$.
 
\end{lemma}
\begin{proof}
   If $m$ is even,  the graph $C_m\times P_2$ is a disjoint union of two $C_m$'s and if $m$ is odd, the graph $C_m\times P_2$ is isomorphic to $C_{2m}$. Thus, in both the cases, we have $\chi_s(C_m\times P_2)=3$. \qed
\end{proof}

\begin{lemma}\label{lem-21}
 For every integer $m$, where $m \geq 3$, we have $\chi_s(C_m \times P_3)=3$.
 
\end{lemma}
\begin{proof}
From Lemma \ref{thm-sylvester}, every positive integer greater than or equal to $12$ can be expressed as an integer linear combination of $4$ and $5$. As first three rows of $C_4 \times P_3$ and $C_5 \times P_3$ are identical (see Fig.~\ref{P3}), by selecting suitable copies of colorings of $C_4 \times P_3$ and $C_5 \times P_3$, we can obtain a $3$-star coloring of $C_m \times P_3$ for $n \geq 12$. As every integer $n \in \{6,8,9,10\}$ can be expressed as an integer linear combination of $3$, $4$ and $5$, we get $3$-star coloring of $C_m \times P_3$ for $n \in \{6,8,9,10\}$. $3$-star colorings of $C_3 \times P_3$, $C_7 \times P_3$ and $C_{11} \times P_3$ are given in the Fig.~\ref{P3}. Therefore we have $\chi_s(C_m \times P_3)\leq 3$. As the graph $C_m\times P_3$ contains $C_4$ as  a subgraph, therefore from Lemma \ref{lem-subgraph} and \ref{lem-c}, we have $\chi_s(C_m\times P_3)\geq \chi_s(C_4) = 3$. Altogether we have 
$\chi_s(C_m \times P_3)=3$. \qed

\begin{figure}[ht]
 \centering
\begin{tikzpicture}[scale=.02, transform shape]
\node [draw, shape=circle] (u1) at (-180,-20) {};
\node [draw, shape=circle] (u2) at (-120,-20) {};
\node [draw, shape=circle] (u3) at (-180,-240) {};
\node [draw, shape=circle] (u4) at (-120,-240) {};

\node [scale=50] at (-170,-30) {1};
\node [scale=50] at (-170,-50) {2};
\node [scale=50] at (-170,-70) {1};
\node [scale=50] at (-170,-90) {1};
\node [scale=50] at (-170,-110) {1};
\node [scale=50] at (-170,-130) {1};
\node [scale=50] at (-170,-150) {1};
\node [scale=50] at (-170,-170) {1};
\node [scale=50] at (-170,-190) {1};
\node [scale=50] at (-170,-210) {1};
\node [scale=50] at (-170,-230) {3};

\node [scale=50] at (-150,-30) {1};
\node [scale=50] at (-150,-50) {2};
\node [scale=50] at (-150,-70) {3};
\node [scale=50] at (-150,-90) {3};
\node [scale=50] at (-150,-110) {2};
\node [scale=50] at (-150,-130) {2};
\node [scale=50] at (-150,-150) {3};
\node [scale=50] at (-150,-170) {3};
\node [scale=50] at (-150,-190) {2};
\node [scale=50] at (-150,-210) {2};
\node [scale=50] at (-150,-230) {3};

\node [scale=50] at (-130,-30) {1};
\node [scale=50] at (-130,-50) {2};
\node [scale=50] at (-130,-70) {1};
\node [scale=50] at (-130,-90) {1};
\node [scale=50] at (-130,-110) {1};
\node [scale=50] at (-130,-130) {1};
\node [scale=50] at (-130,-150) {1};
\node [scale=50] at (-130,-170) {1};
\node [scale=50] at (-130,-190) {1};
\node [scale=50] at (-130,-210) {1};
\node [scale=50] at (-130,-230) {3};

\draw (u1)--(u2)--(u4)--(u3)--(u1);
\node [scale=50] at (-150,-5) {$C_{11} \times P_3$};

\node [draw, shape=circle] (u1) at (-280,-100) {};
\node [draw, shape=circle] (u2) at (-220,-100) {};
\node [draw, shape=circle] (u3) at (-280,-240) {};
\node [draw, shape=circle] (u4) at (-220,-240) {};

\node [scale=50] at (-270,-110) {1};
\node [scale=50] at (-270,-130) {2};
\node [scale=50] at (-270,-150) {1};
\node [scale=50] at (-270,-170) {1};
\node [scale=50] at (-270,-190) {1};
\node [scale=50] at (-270,-210) {1};
\node [scale=50] at (-270,-230) {3};

\node [scale=50] at (-250,-110) {1};
\node [scale=50] at (-250,-130) {2};
\node [scale=50] at (-250,-150) {3};
\node [scale=50] at (-250,-170) {3};
\node [scale=50] at (-250,-190) {2};
\node [scale=50] at (-250,-210) {2};
\node [scale=50] at (-250,-230) {3};

\node [scale=50] at (-230,-110) {1};
\node [scale=50] at (-230,-130) {2};
\node [scale=50] at (-230,-150) {1};
\node [scale=50] at (-230,-170) {1};
\node [scale=50] at (-230,-190) {1};
\node [scale=50] at (-230,-210) {1};
\node [scale=50] at (-230,-230) {3};

\draw (u1)--(u2)--(u4)--(u3)--(u1);
\node [scale=50] at (-250,-85) {$C_7 \times P_3$};

\node [draw, shape=circle] (u1) at (-380,-140) {};
\node [draw, shape=circle] (u2) at (-320,-140) {};
\node [draw, shape=circle] (u3) at (-380,-240) {};
\node [draw, shape=circle] (u4) at (-320,-240) {};

\node [scale=50] at (-370,-150) {1};
\node [scale=50] at (-370,-170) {2};
\node [scale=50] at (-370,-190) {1};
\node [scale=50] at (-370,-210) {1};
\node [scale=50] at (-370,-230) {3};

\node [scale=50] at (-350,-150) {1};
\node [scale=50] at (-350,-170) {3};
\node [scale=50] at (-350,-190) {3};
\node [scale=50] at (-350,-210) {2};
\node [scale=50] at (-350,-230) {2};

\node [scale=50] at (-330,-150) {1};
\node [scale=50] at (-330,-170) {2};
\node [scale=50] at (-330,-190) {1};
\node [scale=50] at (-330,-210) {1};
\node [scale=50] at (-330,-230) {3};

\draw (u1)--(u2)--(u4)--(u3)--(u1);
\node [scale=50] at (-350,-125) {$C_5 \times P_3$};

\node [draw, shape=circle] (v1) at (-480,-160) {};
\node [draw, shape=circle] (v2) at (-420,-160) {};
\node [draw, shape=circle] (v3) at (-480,-240) {};
\node [draw, shape=circle] (v4) at (-420,-240) {};

\node [scale=50] at (-470,-170) {1};
\node [scale=50] at (-470,-190) {2};
\node [scale=50] at (-470,-210) {1};
\node [scale=50] at (-470,-230) {2};

\node [scale=50] at (-450,-170) {1};
\node [scale=50] at (-450,-190) {3};
\node [scale=50] at (-450,-210) {3};
\node [scale=50] at (-450,-230) {2};

\node [scale=50] at (-430,-170) {1};
\node [scale=50] at (-430,-190) {2};
\node [scale=50] at (-430,-210) {1};
\node [scale=50] at (-430,-230) {2};

\draw (v1)--(v2)--(v4)--(v3)--(v1);
\node [scale=50] at (-450,-145) {$C_4 \times P_3$};


\node [draw, shape=circle] (v1) at (-570,-180) {};
\node [draw, shape=circle] (v2) at (-510,-180) {};
\node [draw, shape=circle] (v3) at (-570,-240) {};
\node [draw, shape=circle] (v4) at (-510,-240) {};

\node [scale=50] at (-560,-190) {1};
\node [scale=50] at (-560,-210) {2};
\node [scale=50] at (-560,-230) {3};

\node [scale=50] at (-540,-190) {1};
\node [scale=50] at (-540,-210) {2};
\node [scale=50] at (-540,-230) {3};

\node [scale=50] at (-520,-190) {1};
\node [scale=50] at (-520,-210) {2};
\node [scale=50] at (-520,-230) {3};

\draw (v1)--(v2)--(v4)--(v3)--(v1);
\node [scale=50] at (-540,-165) {$C_3 \times P_3$};

\end{tikzpicture}
  \caption{~$5$-star colorings of $C_m \times P_3$, $m \in \{3,4,5,7,11\}$}\label{P3}
\end{figure}

\end{proof}

\begin{lemma}\label{C3kP45}
 For every pair of positive integers $m$ and $n$, where $m\geq 3$, $m=3k$ for some $k \in \mathbb{N}$ and $n \in \{4,5\}$,  we have $\chi_s(C_m \times P_n)=4$.
\end{lemma}

\begin{proof}

The proof is divided into two cases.

\textbf{Case~1.} When $m = 3$.\\
Consider the graph $C_3 \times P_4$.
Let $V(C_3)  = \{u_1,u_2,u_3\}$, $V(P_4) = \{v_1,v_2,v_3,v_4\}$ and $V(C_3\times P_4) = \{(u_i,v_j) | i\in[3],j\in [4]\}$.
As $C_3 \times P_3$ is a subgraph of $C_3 \times P_4$, from Lemma \ref{lem-subgraph} and \ref{lem-21}, we have $\chi_s(C_3\times P_4) \geq \chi_s(C_3 \times P_3) = 3$. We have observed that the graph $C_3 \times P_3$ has a unique (up to permutation of colors) $3$-star coloring, which is given in Fig.~\ref{C3P3}. Suppose $\chi_s(C_3\times P_4) = 3$ and let $f$ be a 3-star coloring of $C_3 \times P_4$ with colors $a,b,c$. Then from the above observation, $f$ restricted to the vertices of subgraph $C_3 \times P_3$, gives a coloring as shown in the  Fig.~\ref{C3P3}.  Now consider the vertex $(u_1,v_4)$ of $C_3\times P_4$. Clearly, $f((u_1,v_4)) \notin \{b,c\}$, else $f$ is not proper coloring. Also $f((u_1,v_4)) \neq a$, else we get a bicolored path of length three. Therefore, $f((u_1,v_4)) \notin \{a,b,c\}$, which is a contraction to our assumption that $f$ is a 3-star coloring of $C_3 \times P_4$. Thus  $\chi_s(C_3\times P_4) \geq 4$. Since the graph $C_3\times P_4$ is a subgraph of $C_3 \times P_5$, we have $\chi_s(C_3\times P_5) \geq 4$. $4$-star colorings of $C_3\times P_4$ and $C_3\times P_5$ are given in Fig.~\ref{C3P45}. Therefore, $\chi_s(C_3\times P_n) = 4$, $n\in \{4,5\}$.

\begin{figure}[ht]
 \centering
\begin{tikzpicture}[scale=.02, transform shape]
\node [draw, shape=circle] (v1) at (-300,-20) {};
\node [draw, shape=circle] (v2) at (-240,-20) {};
\node [draw, shape=circle] (v3) at (-300,-80) {};
\node [draw, shape=circle] (v4) at (-240,-80) {};

\node [scale=50] at (-290,-30) {$a$};
\node [scale=50] at (-290,-50) {$b$};
\node [scale=50] at (-290,-70) {$c$};

\node [scale=50] at (-270,-30) {$a$};
\node [scale=50] at (-270,-50) {$b$};
\node [scale=50] at (-270,-70) {$c$};

\node [scale=50] at (-250,-30) {$a$};
\node [scale=50] at (-250,-50) {$b$};
\node [scale=50] at (-250,-70) {$c$};

\draw (v1)--(v2)--(v4)--(v3)--(v1);
\end{tikzpicture}
  \caption{~A $3$-star coloring of $C_3 \times P_3$}\label{C3P3}
\end{figure}

\textbf{Case~2.} When $m >3$.\\
For $m >3$ and $n\in \{4,5\}$, the graph $C_m\times P_n$ contains $P_2 \square P_3$ as a subgraph, thus from Lemma \ref{lem-subgraph} and \ref{lem-1}, we have $\chi_s(C_m\times P_n) \geq \chi_s(P_2 \square P_3) = 4$. For $n \in \{4,5\}$, $4$-star coloring of $C_{3k} \times P_n$ can be obtained by using suitable copies of colorings of  $C_3 \times P_n$, $n\in \{4,5\}$. Therefore, $\chi_s(C_m\times P_n)=4$, for $n \in \{4,5\}$.  \qed
 \end{proof}

\begin{figure}[ht]
 \centering
\begin{tikzpicture}[scale=.02, transform shape]
  \node [draw, shape=circle] (u1) at (-190,-20) {};
\node [draw, shape=circle] (u2) at (-90,-20) {};
\node [draw, shape=circle] (u3) at (-190,-80) {};
\node [draw, shape=circle] (u4) at (-90,-80) {};

\node [scale=50] at (-180,-30) {$1$};
\node [scale=50] at (-180,-50) {$1$};
\node [scale=50] at (-180,-70) {$1$};

\node [scale=50] at (-160,-30) {$2$};
\node [scale=50] at (-160,-50) {$3$};
\node [scale=50] at (-160,-70) {$4$};

\node [scale=50] at (-140,-30) {$2$};
\node [scale=50] at (-140,-50) {$3$};
\node [scale=50] at (-140,-70) {$4$};

\node [scale=50] at (-120,-30) {$2$};
\node [scale=50] at (-120,-50) {$3$};
\node [scale=50] at (-120,-70) {$4$};

\node [scale=50] at (-100,-30) {$1$};
\node [scale=50] at (-100,-50) {$1$};
\node [scale=50] at (-100,-70) {$1$};

\draw (u1)--(u2)--(u4)--(u3)--(u1);
\node [scale=50] at (-150,-5) {$C_3 \times P_5$};

\node [draw, shape=circle] (u1) at (-370,-20) {};
\node [draw, shape=circle] (u2) at (-290,-20) {};
\node [draw, shape=circle] (u3) at (-370,-80) {};
\node [draw, shape=circle] (u4) at (-290,-80) {};

\node [scale=50] at (-360,-30) {$1$};
\node [scale=50] at (-360,-50) {$1$};
\node [scale=50] at (-360,-70) {$1$};

\node [scale=50] at (-340,-30) {$2$};
\node [scale=50] at (-340,-50) {$3$};
\node [scale=50] at (-340,-70) {$4$};

\node [scale=50] at (-320,-30) {$2$};
\node [scale=50] at (-320,-50) {$3$};
\node [scale=50] at (-320,-70) {$4$};

\node [scale=50] at (-300,-30) {$2$};
\node [scale=50] at (-300,-50) {$3$};
\node [scale=50] at (-300,-70) {$4$};

\draw (u1)--(u2)--(u4)--(u3)--(u1);
\node [scale=50] at (-330,-5) {$C_3 \times P_4$};

\end{tikzpicture}
  \caption{~$4$-star colorings of $C_3 \times P_4$ and $C_3 \times P_5$}\label{C3P45}
\end{figure}

\begin{lemma} \label{lem-33}
     For every pair of positive integers $m$ and $n$, where $m\neq 3k$ for some $k \in \mathbb{N}$ and $n \in \{4,5\}$,  we have $4 \leq \chi_s(C_m \times P_n) \leq 5$.
\end{lemma}
\begin{proof}
For $m >3$ and $n \in \{4,5\}$, the graph $P_2 \square P_3$ is a subgraph of 
$C_m \times P_n$.  Thus from Lemma \ref{lem-subgraph} and \ref{lem-1} we have $\chi_s(C_m \times P_n) \geq 4$. As the graph $C_m \times P_n$ is a subgraph of $C_m \times C_n$, thus from Lemma \ref{lem-subgraph} and Theorem \ref{th-CC}, we have $\chi_s(C_m \times P_n) \leq 5$. \qed
\end{proof}
 
\begin{lemma}
For every positive integer $n \geq 6$, we have $\chi_s(C_{3} \times P_n)=5$.
\end{lemma}
 \begin{proof}
For $n\geq 6$, the graph $C_3 \times P_n$ has $C_3 \times P_5$ as a subgraph. Thus from Lemma \ref{lem-subgraph} and \ref{C3kP45}, we have $\chi_s(C_{3} \times P_n) \geq \chi_s(C_{3} \times P_5) = 4$. We have observed that the graph $C_3 \times P_5$ has a unique coloring (up to permutation of colors) pattern with four colors $a,b,c,d$ as shown in the Fig.~\ref{C3P5}. By using arguments similar to Case 1 of Lemma \ref{C3kP45}, we can show that for $n\geq 6$, $\chi_s(C_{3} \times P_n) \geq 5$. Also $C_3 \times P_n$ is a subgraph of $C_3 \times C_n$, for $n\geq 6$, therefore from Lemma \ref{th-c3cn}, we have $\chi_s(C_3\times P_n) \leq \chi_s(C_3 \times C_n) = 5$. Altogether we have $\chi_s(C_3 \times P_n) = 5$. \qed

 \end{proof}
 
\begin{figure}[ht]
 \centering
\begin{tikzpicture}[scale=.02, transform shape]
  \node [draw, shape=circle] (u1) at (-190,-20) {};
\node [draw, shape=circle] (u2) at (-90,-20) {};
\node [draw, shape=circle] (u3) at (-190,-80) {};
\node [draw, shape=circle] (u4) at (-90,-80) {};

\node [scale=50] at (-180,-30) {$a$};
\node [scale=50] at (-180,-50) {$a$};
\node [scale=50] at (-180,-70) {$a$};

\node [scale=50] at (-160,-30) {$b$};
\node [scale=50] at (-160,-50) {$c$};
\node [scale=50] at (-160,-70) {$d$};

\node [scale=50] at (-140,-30) {$b$};
\node [scale=50] at (-140,-50) {$c$};
\node [scale=50] at (-140,-70) {$d$};

\node [scale=50] at (-120,-30) {$b$};
\node [scale=50] at (-120,-50) {$c$};
\node [scale=50] at (-120,-70) {$d$};

\node [scale=50] at (-100,-30) {$a$};
\node [scale=50] at (-100,-50) {$a$};
\node [scale=50] at (-100,-70) {$a$};

\draw (u1)--(u2)--(u4)--(u3)--(u1);

\end{tikzpicture}
  \caption{~A $4$-star coloring of $C_3 \times P_5$}\label{C3P5}
\end{figure}

\begin{lemma}
 For every positive integer $n \geq 4$, we have $\chi_s(C_4 \times P_n)=5$.
\end{lemma}
\begin{proof}
 For $n\geq 4$, the graph $C_4\times P_n$ contains $Y$ as a subgraph, thus from Lemma \ref{lem-subgraph}, we have $\chi_s(C_4\times P_n) \geq \chi_s(Y) = 5$. Since $C_4 \times P_n$ is a subgraph of $C_4 \times C_n$, therefore from Lemma \ref{th-c4cn}, we have $\chi_s(C_4\times P_n) \leq \chi_s(C_4\times C_n) = 5$. Altogether we have $\chi_s(C_4\times P_n) = 5$.  \qed
\end{proof}

\begin{lemma}
     For every positive integer $n \geq 4$, we have $\chi_s(C_5 \times P_n)=5$.
\end{lemma}

\begin{proof}
By case by case analysis we found that $\chi_s(C_5\times P_4) = 5$.
  As the graph $C_5 \times P_4$ is a subgraph of $C_5\times P_n$ for $n\geq 5$, from Lemma \ref{lem-subgraph} we have $\chi_s(C_5 \times P_n) \geq \chi_s(C_5 \times P_4) = 5$. Since $C_5 \times P_n$ is a subgraph of $C_5 \times C_n$, thus from Lemma \ref{lem-subgraph} and \ref{th-c5cn}, we have $\chi_s(C_5\times P_n) \leq \chi_s(C_5\times C_n) = 5$. Altogether we have $\chi_s(C_5\times P_n) = 5$.  \qed
\end{proof}

\begin{lemma}
    For every positive integer $n \geq 6$, we have $\chi_s(C_6 \times P_n)=5$. 
\end{lemma}
\begin{proof}
    For $n \geq 8$, the graph $C_{6}\times P_n$ contains $Z$ as a subgraph. Thus from Lemma \ref{lem-subgraph}, we have $\chi_s(C_6 \times P_n) \geq \chi_s(Z) = 5$. 
    Since $C_6 \times P_n$ is a subgraph of $C_6 \times C_n$, therefore from Lemma \ref{th-C6C8C9C10}, we have $\chi_s(C_6\times P_n) \leq \chi_s(C_6\times C_n) = 5$. Altogether, for $n \geq 8$, we have $\chi_s(C_6\times P_n) = 5$. Also by tedious case by case analysis we found that $\chi_s(C_6 \times P_6) = 5$ and $\chi_s(C_6 \times P_7) = 5$. \qed
\end{proof}

\begin{lemma}
    For every positive integer $n \geq 4$, we have $\chi_s(C_7 \times P_n)=5$.
\end{lemma}
\begin{proof}
    For $n\geq 7$, the graph $C_{7}\times P_n$ contains $P_4 \square P_4$ as a subgraph. Thus from Lemma \ref{lem-subgraph} and \ref{lem-1}, we have $\chi_s(C_7 \times P_n) \geq \chi_s(P_4 \square P_4) = 5$. Since $C_7 \times P_n$ is a subgraph of $C_7 \times C_n$, thus from Lemma \ref{lem-subgraph} and \ref{C7Cn}, we have $\chi_s(C_7\times P_n) \leq \chi_s(C_7\times C_n) = 5$. Altogether, for $n \geq 7$, we have $\chi_s(C_7\times P_n) = 5$. Also by tedious case by case analysis we found that $\chi_s(C_7 \times P_4) = \chi_s(C_7 \times P_5)=\chi_s(C_7 \times P_6)=5$. \qed
\end{proof}

\begin{lemma}
For every pair of positive integers $m$ and $n$, where $m \geq 8$, $n \geq 6$, we have 
$\chi_s(C_m \times P_n)= 5$.
\end{lemma}
 
\begin{proof}
For $m \geq 8$ and $n \geq 6$, the graph $C_{m}\times P_n$ contains $Z$ as a subgraph. Thus from Lemma \ref{lem-subgraph}, we have  $\chi_s(C_m \times P_n) \geq \chi_s(Z) = 5$. Since $C_m \times P_n$ is the subgraph of $C_m \times C_n$, therefore from Lemma \ref{lem-subgraph}, \ref{th-C6C8C9C10} and \ref{th-C11}, we have $\chi_s(C_m\times P_n) \leq \chi_s(C_m\times C_n) = 5$ for $m\geq 8$ and $n \geq 6$. Altogether we have $\chi_s(C_m\times P_n) = 5$. \qed

\end{proof}

\bibliographystyle{plain}
\bibliography{ref.bib}

\end{document}